\documentclass[11pt]{article}
\usepackage[a4paper]{geometry}
\usepackage{amsmath,amsthm,amssymb}
\usepackage{color,latexsym}
\usepackage[shortlabels]{enumitem}
\usepackage{amsmath,amssymb,amsfonts,amssymb}
\usepackage{graphicx}
\usepackage{float}
\usepackage[authoryear]{natbib}
\usepackage{centernot}
\usepackage{mathrsfs}
\usepackage{amsthm}
\usepackage{ulem}
\usepackage{authblk}
\usepackage{autobreak}
\theoremstyle{plain} 
\newtheorem{theo}{Theorem}[section]   
\AfterEndEnvironment{theorem}{\@aftertheorem}
\newtheorem{coro}[theo]{Corollary}
\newtheorem{lem}{Lemma}[section]
\theoremstyle{definition}

\newtheorem{assu}{Assumption}[section]
\newtheorem{remark}{Remark}[section]

\newcommand{\e}{\mathbb{E}}

\newcommand{\p}{\mathbb{P}}

\newcommand{\R}{\mathbb{R}}

\newcommand{\bdx}{\boldsymbol{x}}

\newcommand{\bdxi}{\boldsymbol{\xi}}

\newcommand{\bdS}{\boldsymbol{S}}
\newcommand{\bdbe}{\boldsymbol{\beta}}
\newcommand{\bbdbe}{\boldsymbol{\bar{\beta}}}
\newcommand{\hbdbe}{\boldsymbol{\hat{\beta}}}

\newcommand{\sgn}{\operatorname{sign}}
\usepackage {chngcntr}
\newenvironment{proof*}[1][Proof]{\par
  \pushQED{}
  \normalfont
  \topsep6\p@\@plus6\p@\relax
  \trivlist
  \item[\hskip\labelsep\itshape #1\@addpunct{.}]
}{\popQED\endtrivlist\@endpefalse}
\numberwithin{equation}{section}
\usepackage{color}
\usepackage{graphicx}
\usepackage{hyperref}
\hypersetup{colorlinks=true,
            linkcolor=blue,
            anchorcolor=blue,
            citecolor=blue}
\usepackage{cleveref}
\crefname{lem}{Lemma}{Lemmas}

        \title{\bf\LARGE{Berry--Esseen Bounds and Moderate Deviations for Catoni-Type Robust Estimation \footnote{Equal Contributions}}} 
        \author{ZhiJun Cai\thanks{Department of Statistics and Data Science, Southern University of Science and Technology, ShenZhen, Chain. Email: 122312911@mail.sustech.edu.cn} \and Xiang Li\thanks{Department of Statistics and Data Science, Southern University of Science and Technology, ShenZhen, Chain. Email: lixiang3@sustech.edu.cn} \and Lihu Xu\thanks{Department of Statistics and Probability, Michigan State University, East Lansing, MI, 48824, USA.\\
        Department of Mathematics, Faculty of Science and Technology, University of Macau, Avenida da Universidade
        Taipa, Macau, China. Email: xulihu@msu.edu}}
        \date{ }
\begin{document}
\maketitle

\begin{abstract}
A powerful robust mean estimator introduced by Catoni (2012) allows for mean estimation of heavy-tailed data while achieving the performance characteristics of classical mean estimator for sub-Gaussian data. While Catoni’s framework has been widely extended across statistics, stochastic algorithms, and machine learning, fundamental asymptotic questions regarding the Central Limit Theorem and rare event deviations remain largely unaddressed. In this paper, we investigate Catoni-type robust estimators in two contexts: (i) mean estimation for heavy-tailed data, and (ii) linear regression with heavy-tailed innovations. For the first model, we establish the Berry--Esseen bound and moderate deviation principles, addressing both known and unknown variance settings. For the second model, we demonstrate that the associated estimator is consistent and satisfies a multi-dimensional Berry-Esseen bound.   
\end{abstract}


\tableofcontents 
\section{Introduction}
Parameter estimation is a fundamental task in statistics. Among such tasks, estimating the mean of a distribution is  one of the most basic and central problems. Let $X_1, \dots, X_n$ be independent and identically distributed (i.i.d.) random variables with expectation $\mu$, the parameter of interest. A standard estimator of $\mu$ is the sample mean, defined by
$
\bar{X} = \frac{1}{n} \sum_{i=1}^n X_i.
$
However, the sample mean is highly sensitive to extreme values and may be suboptimal for heavy-tailed distributions. To mitigate the variance inflation and instability caused by outliers under heavy tails, \cite{Huber1964} introduced M-estimators based on the Huber loss, striking a tunable balance between efficiency and robustness. Building on this line, 
\cite{Catoni2012Challenging} proposed a new robust estimator of the mean, defined as the solution $\hat{\theta}$ to
\begin{align}\label{eq-t-02}
\sum_{i=1}^n \varphi\big[\alpha_{n}(X_i - \theta)\big] = 0,
\end{align}
where $\alpha_{n}$ is a positive tuning parameter,  and $\varphi(x)$ is a non-decreasing and continuous function satisfying that \begin{align}\label{eq-pa-01} -\log(1-x+x^{2}/2)  \leq  \varphi(x)\leq \log(1+x+x^{2}/2),\quad \forall x\in \R. 
\end{align}
Importantly, \cite{Catoni2012Challenging} developed a finite-sample analytical framework for robust mean estimation, establishing nonasymptotic guarantees that highlight the advantages of robust estimators and the limitations of the sample mean under heavy tails. 
\par
Since Catoni's seminal work, many related studies have appeared, including applications to empirical risk minimization under unbounded losses \citep{Brownlees-Joly-Lugosi2015Empirical}, investigations of cases with finite $1+\delta$ ($0<\delta<1$) moments \citep{chen-Jin-Li-Xu2019generalized}, extensions to confidence sequences \citep{wang2023,LiWang2022}, as well as other related works, see \cite{Huber2019} for stochastic algorithm, \cite{sun-zhou-and-fan,FaLiWa2017} for high dimensional statistics, and \cite{XYYZ2023} for robust estimation in the deep neural network.
\par
However, despite the many subsequent studies, the asymptotic properties of Catoni-type estimators have rarely been investigated. \cite{Yao2022} analyzed the asymptotic behavior of $\hat{\theta}$ in a particular setting. Specifically, they considered the special case $\varphi(x) = \varphi_{1}(x): = \text{sign}(x)\log(1+|x|+x^{2}/2)$ under the assumptions $X_{1}-u \stackrel{d}{=} u-X_{1}$ and $\alpha_{n} = O(n^{-1/2})$, and proved that
\begin{align}\label{eq-pa-02}
  \frac{1}{\sqrt{\e [\varphi_{1}^{2}(\alpha_{n}(X_{1}-u))]/\alpha_{n}^{2}}} 
  \sqrt{n}(\hat{\theta}-u) \stackrel{d}{\longrightarrow} N(0,1).
\end{align}
However, the symmetry assumption is overly restrictive, and qualitative results such as the central limit theorem offer limited guidance for quantitative calibration or tuning. In fact, in the absence of symmetry, the bias induced by Catoni-type estimators becomes apparent (see the main theorems in Section \ref{section:main results}). Consequently, once the symmetry assumption is removed, further quantitative analyses such as Berry–Esseen bounds and moderate deviation results become considerably more valuable.

Berry–Esseen bounds and moderate deviation theory offer practical, finite-sample guidance for Gaussian-style inference. Berry–Esseen bounds tell us how accurate normal approximations are for standardized estimators and common test statistics, so we can judge confidence-interval coverage and test size in real samples. Moderate deviation results describe mid-tail probabilities beyond the core CLT range, improving p-values near typical cutoffs and signaling when normal critical values remain trustworthy or need slight adjustment. Recent studies across diverse settings attest to their sustained centrality in statistical research—for example, non-normal approximation \citep{shao-and-zhang2019,ShaoZhangZhang2023};  high-dimensional normal approximation \citep{Kato2013,Kato-Koike2022Improved,fangandkoike2024}; self-normalized processes
\citep{shao-and-jing(2003),shao-zhou2016,gao-shao-shi2022-refined}; exponential random graph models \citep{ChatterjeeDiaconis2013, Fang-Liu-Shao-zhao2025-exponential}.

In this paper, we establish the asymptotic properties of $\hat{\theta}$ for arbitrary functions satisfying \eqref{eq-pa-01}, \emph{without} the condition $X_{1}-u \stackrel{d}{=} u-X_{1}$. Our results include  Berry–Esseen bounds as well as Cram\'{e}r-type moderate deviation results. 
At the same time, recognizing that the robustness tuning parameter in Catoni’s estimator can be inconvenient to select when the variance is unknown, we also study Cramér-type moderate deviations for a self-normalized formulation of Catoni’s estimator. Specifically, the estimator we analyze is $\hat{\theta}_{s}$ defined as the solution to the following equation
\begin{align}\label{eq-pq-02}
  \frac{\hat{\sigma}}{n a_{n}}\sum_{i=1}^{n}\varphi\Big[\frac{X_{i}-\theta}{\hat{\sigma}}a_{n}\Big]=0,
\end{align}
with $ \hat{\sigma}^{2}=\frac{1}{n-1}\sum_{i=1}^{n} (X_{i}-\bar{X})^{2}$ and tuning parameter $a_{n}$.


\par

As a generalization of mean estimation, we also consider the following regression model. 
Let $(y_1,\bdx_1),\dots,(y_n,\bdx_n)\in \mathbb{R}\times \mathbb{R}^p$ be independent data samples such that
\begin{align}\label{eq-p-01}
  y_i = \bdx'_i \boldsymbol{\beta}^* + \varepsilon_i, \quad 1\le i\le n,
\end{align}
where $\varepsilon_1, \dots, \varepsilon_n$ are independent random variables with $\mathbb{E}[\varepsilon_i] = 0$ and $\mathbb{E}[\varepsilon_i^2] = \sigma_i^2 < \infty$, and are independent of $\{\bdx_1, \dots, \bdx_n\}$.
Here, the parameter of interest is the regression coefficient vector $\boldsymbol{\beta}^*$. 
We estimate $\boldsymbol{\beta}^*$ by solving the estimating equation
\begin{align}\label{eq-p-03}
h(\bdbe):= \frac{1}{n\alpha_{n}}\sum_{i=1}^{n}\bdx_{i}\,\varphi\!\left[\alpha_{n}\bigl(y_{i}-\bdx'_{i} \boldsymbol{\beta}\bigr)\right] = 0,
\end{align}
where $\alpha_{n}$ is a sequence of constants depending on $n$ and $\varphi$ is a function satisfying (\ref{eq-pa-01}). 
We denote the resulting estimator by $\hbdbe$. In this paper, we establish both non-asymptotic and asymptotic results for $\hbdbe-\bdbe^{*}$. 

Our main contributions can be summarized as follows: (i) For mean estimation, we derive general Berry--Esseen bounds and Cram\'{e}r-type moderate deviation results. Notably, unlike classical studies that require certain restrictions on the derivative of $\e\{\varphi[\alpha_{n}(X_{1}-\cdot\,)]\}$ (e.g., \cite{bentkus1997berry,JURECKOVA1988191}), we do not assume any such conditions. When the variance is unknown, we also establish Berry--Esseen bounds and moderate deviation results for the corresponding self-normalized statistics, under the assumption of a finite $2+\delta$ moment.
 (ii) We extend the Catoni-type mean estimation to regression models and study the properties of this estimator. The results show that the deviations of such estimator have the same order as in the Gaussian sample case. Moreover, under some additional mild conditions, we  also develop the asymptotic theory of this estimator.

\par
The rest of this paper is organized as follows. Section \ref{section:main results} presents our main results in two parts. 
The first part focuses on asymptotic results for mean estimation, including Berry--Esseen bounds and Cram\'{e}r-type moderate deviations for both known and unknown variance. 
The second part covers results for Catoni-type regression, including   non-asymptotic results and  Berry--Esseen bounds. 
Section \ref{section:proofs-of-mean-estimation} contains the proofs of main results for mean estimation, while the proofs of main results for Catoni-type robust regression are postponed to  Section \ref{section-proof-of-catoni-type}.

\section{Main Results}\label{section:main results}
In this section, we present our results, which are divided into two parts: mean estimation and regression models.

\subsection{Berry--Esseen bounds (BEDs) and moderate deviations (MDs) for mean estimation}
\subsubsection{BEDs and MDs for mean estimation with known variance}
In this subsection, we focus on $\hat{\theta}$, which is defined as the solution to equation (\ref{eq-t-02}). 
 For any fixed $n$, let $u_{n}$ be any solution to equation 
\begin{align*}
  f(x)=\e\{\varphi[\alpha_{n}(X_{1}-x)]\}=0.
\end{align*}
We make the following assumption.
\begin{assu}\label{assu-mean-CLT}
   We assume that 
  \begin{enumerate}[(a)] 
  \item \label{assu-mean-clt-a} $\varphi(x)$ is a  continuous  and non-decreasing function satisfying that
  \begin{align}\label{eq-t-04}
     -\log\big(1-x+x^{2}/2\big)  \leq  \varphi(x)\leq \log\big(1+x+x^{2}/2\big)\quad \text{for any } x\in \R.
      \end{align}
  \item \label{assu-clt-b} Let $\alpha_{n}=a_{n}\sigma^{-1}$ 
  where $\sigma^{2}=\mathrm{Var}(X_{1}), a_{n}>0$. Then there exists  a non-negative constant $C_{0}<\infty$ such that
  \begin{align*}
    \limsup_{n\to \infty}\sqrt{n}a_{n}\leq C_{0}.
  \end{align*}
\end{enumerate}
\end{assu}
The following theorem is one of our main results.
\begin{theo}\label{thm-mean-CLT}
  Under Assumption \ref{assu-mean-CLT},   we have
  \begin{align}\label{eq-thm-mean-CLT}
  \sup_{z\in \R}\big|\p\big(\sqrt{n}(\hat{\theta}-u_{n})/\sigma\leq z\big)-\Phi(z)\big|\leq C\beta_{2}+C\beta_{3},
  \end{align}
  where $C$ is a positive  constant depending on the distribution of $X_{1}$ and 
\[
\beta_{2}=\frac{1}{\sigma^{2}}\e\big\{ |X_{1}-u|^{2}\mathbf{1}(|X_{1}-u|\geq \sqrt{n}\sigma)\big\},\quad \beta_{3}=\frac{1}{\sqrt{n}\sigma^{3}}\e\big\{ |X_{1}-u|^{3}\mathbf{1}(|X_{1}-u|\leq \sqrt{n}\sigma)\big\}.
\]
\end{theo}
\begin{remark}
  We first make some remarks on $u_{n}.$ Here, we subtract $u_{n}$ rather than the expectation of $X_{1}$ when performing the standardization. This is because 
$\hat{\theta}$ is a biased estimator of $u$ (at least in most cases). More precisely, similar to the case of $M$-estimators, the estimator 
$\hat{\theta}$ should be regarded as an estimate of $u_{n}.$ In some special cases, such as when condition \( X_{1}-u\stackrel{d}{=}u-X_{1} \) holds and $\varphi(x)=-\varphi(-x)$, we have \( u_{n} = u \). In the general case, however, \( u_{n} \neq u \), and Lemma \ref{lem-2} implies that
\begin{align}\label{eq-t-07}
 |u_{n}-u|\leq \frac{\alpha_{n}\sigma^{2}}{\sqrt{1-\alpha_{n}^{2}\sigma^{2}}}.
\end{align}
\end{remark}
Observe that for any $0<\varepsilon<1$, we have
\begin{align}\label{eq-nn-0}
  \beta_{2}+\beta_{3}
  &=\frac{1}{\sigma^{2}}\e\big\{ |X_{1}-u|^{2}\mathbf{1}(|X_{1}-u|\geq \sqrt{n}\sigma)\big\}\nonumber\\
  &\quad+\frac{1}{\sqrt{n}\sigma^{3}}\e\big\{ |X_{1}-u|^{3}\mathbf{1}(|X_{1}-u|\leq \varepsilon\sqrt{n}\sigma)\big\}\nonumber\\
  &\quad+
  \frac{1}{\sqrt{n}\sigma^{3}}\e\big\{ |X_{1}-u|^{3}\mathbf{1}(\varepsilon\sqrt{n}\sigma< |X_{1}-u|\leq \sqrt{n}\sigma)\big\}
  \nonumber\\
  &\leq \frac{1}{\sigma^{2}}\e\big\{ |X_{1}-u|^{2}\mathbf{1}(|X_{1}-u|> \varepsilon\sqrt{n}\sigma)\big\}+\varepsilon.
\end{align}
Since $\varepsilon$ is arbitrary, then by Lebesgue dominated convergence theorem, $\beta_{2}+\beta_{3}\to 0$ as $n\to \infty.$ So we immediately have the following corollary. 
\begin{coro}\label{coro-mean-clt} Under Assumption \ref{assu-mean-CLT},   we have
  \[
\sqrt{n}(\hat{\theta}-u_{n})/\sigma\xrightarrow{d} N(0,1)\quad \text{as}\quad n\to\infty.
\]
Moreover, if $\sqrt{n}a_{n}\to 0$ as $n\to \infty,$ then 
  \[
\sqrt{n}(\hat{\theta}-u)/\sigma\xrightarrow{d} N(0,1) \quad \text{as}\quad n\to\infty.
\]
\end{coro}
\begin{remark}
Using the fact that $|\varphi(x)|\leq x$ and $|\varphi(x)-x|\leq x^{2}$ yields that $|\varphi^{2}(x)-x^{2}|\leq 4x^{2}\mathbf{1}(|x|\geq 1)+2x^{3}\mathbf{1}(|x|\leq 1)$, which further implies that
\begin{align}\label{eq-t-3.1}
    &\big|\e [\varphi^{2}(\alpha_{n}(X_{1}-u))]/\alpha_{n}^{2}-\sigma^{2}\big|\nonumber\\
    &\quad\leq 4\e\big\{|X_{1}-u|^{2}\mathbf{1}(\alpha_{n}|X_{1}-u|\geq 1)\big\}+2\alpha_{n}\e\big\{|X_{1}-u|^{3}\mathbf{1}(\alpha_{n}|X_{1}-u|\leq 1)\big\}\nonumber\\
    &\quad\longrightarrow 0 \quad \text{as } n\to \infty.
\end{align}
From this perspective, our result covers the result of \cite{Yao2022}.
\end{remark}

Combine Theorem \ref{thm-mean-CLT} and (\ref{eq-t-07}), we obtain the following corollary.
\begin{coro}\label{coro-mean-BE} Under Assumption \ref{assu-mean-CLT}, if $\e|X_{1}-u|^{2+\delta}<\infty$ for some $\delta\in (0,1]$, then 
  \begin{align*}
  \sup_{z\in \R}\big|\p\big(\sqrt{n}(\hat{\theta}-u_{n})/\sigma\leq z\big)-\Phi(z)\big|\leq C n^{-\delta/2}\sigma^{-2-\delta}\e|X_{1}-u|^{2+\delta},
  \end{align*}
  and 
    \begin{align*}
  \sup_{z\in \R}\big|\p\big(\sqrt{n}(\hat{\theta}-u)/\sigma\leq z\big)-\Phi(z)\big|\leq C n^{-\delta/2}\sigma^{-2-\delta}\e|X_{1}-u|^{2+\delta}+C\sqrt{n}a_{n},
  \end{align*}
   where $C$ is a positive  constant depending on the distribution of $X_{1}$.
\end{coro}
Theorem \ref{thm-mean-CLT} establishes a Berry--Esseen bound for $\hat{\theta}$, 
theoretically demonstrating that its distribution can be approximated by a normal distribution.
However, the accuracy of this result for the tail probability $\p\big(\sqrt{n}(\hat{\theta}-u_{n})/\sigma> z\big)$ remains limited. 
To address this, we further present a Cramér-type moderate deviation result, 
which yields a more precise estimate for $\p\big(\sqrt{n}(\hat{\theta}-u_{n})/\sigma> z\big)$.
  \begin{theo}\label{thm-mean-md}
    Under Assumption \ref{assu-mean-CLT}, suppose that there exist $t_{0}>0$, $C_{1}>0$ such that 
    \[
    \e e^{t_{0}\sqrt{|X_{1}|}}<C_{1}<\infty,
    \]
  then there exist two positive constant $c_{0}$ and $C$  depending on $t_{0}$, $C_{0}$ and $C_{1}$ such that
    \begin{align}\label{eq-thm-mean-md-1}
      &\Big|\frac{\p\big(\sqrt{n}(\hat{\theta}-u_{n})/\sigma> z\big)}{1-\Phi(z)}-1\Big|\leq C(1+z^{3})n^{-1/2}
    \end{align}
    for $0\leq z\leq c_{0}n^{1/6}$.
  \end{theo}
 Replacing $u_{n}$ with $u$ leads to the following corollary.
  \begin{coro}\label{coro-mean-MD}
 Keep the same assumptions and notations as in Theorem \ref{thm-mean-md},  then  there exist two positive constant $c'_{0}$ and $C$  depending on $t_{0}$ and $C_{1}$ such that
    \begin{align*}
      &\Big|\frac{\p\big(\sqrt{n}(\hat{\theta}-u)/\sigma> z\big)}{1-\Phi(z)}-1\Big|\leq C(1+z^{3})n^{-1/2}+C(1+z)\sqrt{n}a_{n}
    \end{align*}
    for $0\leq z\leq c'_{0}\min\{n^{1/6}, n^{-1/2}a_{n}^{-1}\}$.
\end{coro}
\subsubsection{BEDs and MDs for mean estimation with unknown variance}
In this subsection, we focus on $\hat{\theta}_{s}$, which is defined as the solution of (\ref{eq-pq-02}). We have the following Berry--Esseen bounds for $\hat{\theta}_{s}$.
\begin{theo}\label{thm-self-mean-clt} Under condition \ref{assu-mean-clt-a}, we have
\begin{align}\label{eq-thm-self-mean-clt-01}
  \sup_{z\in \R} \big|\p\big(\sqrt{n}(\hat{\theta}_{s}-u)/\hat{\sigma}\leq z \big)-\Phi(z)\big|\leq C(\beta_{2}+\beta_{3})+C\sqrt{n}a_{n},
\end{align}
where $C$ is an absolute constant and $\beta_{2}$ and $\beta_{3}$
 are as defined in Theorem \ref{thm-mean-CLT}.
\end{theo}

The following theorem provides a moderate deviation result for 
$\hat{\theta}_{s}$ under a finite $2+\delta$ moment.
\begin{theo}\label{thm-self-mean-md} Suppose that condition \ref{assu-mean-clt-a} holds  and $\e|X_{1}|^{2+\delta}<\infty$ for some $\delta\in (0,1]$. Let 
\[\gamma_{n}=\min\{n^{\delta/(4+2\delta)}, n^{-1/2}a_{n}^{-1}\}
\quad \text{and}\quad d_{k}=\sigma/(\e|X_{1}-u|^{k})^{1/k} \quad\text{for any } k\geq 2,
\]
then
\begin{align}\label{eq-thm-mean-md-01}
  &\Big| \p\big(\sqrt{n}(\hat{\theta}_{s}-u)/\hat{\sigma}\geq z\big)\big/(1-\Phi(z))-1\Big|\nonumber\\
  &\qquad\quad\qquad\leq C\big((1+z)^{2+\delta}n^{-\delta/2}+(1+z)\sqrt{n}a_{n}\big)d_{2+\delta}^{-(2+\delta)}
\end{align}
for $0\leq z\leq c_{0}\gamma_{n}d_{2+\delta}$, where $c_{0}$ and $C$ are positive constants depending on $\delta$ and $\e|X_{1}|^{2+\delta}$.
\end{theo}
\subsection{Multi-dimensional Catoni type robust regression}
\begin{assu}\label{assu-low-dimension-gram-matrix}
   The empirical Gram matrix $\bdS_{n}:=\frac{1}{n}\sum_{i=1}^{n}\bdx_{i}\bdx_{i}'$ is nonsingular. Moreover, there exist two positive constants $c_{l}\leq \lambda_{\min}(\bdS_{n})\leq \lambda_{\max}(\bdS_{n})\leq c_{u}.$ 
\end{assu}
\begin{theo}\label{thm-nonasymtotic-theory}
   Suppose that Assumption  \ref{assu-low-dimension-gram-matrix} holds. Let $L_{n}:=\max _{1 \leq i \leq n}\left\|\boldsymbol{x}_i\right\|_{2}$. For any $\varepsilon \in(0, 1)$, assume 
   \begin{align}\label{eq-thm-nonasymtotic-assu}
       \Delta^{2}=1-\frac{L_{n}^{2}}{c_{l}^{2}}\Big(\alpha_{n}^{2}c_{u}\bar{\sigma}^{2}+\frac{2c_{u}\log(\varepsilon^{-1})}{n} \Big)\geq 0,
   \end{align}
where $\bar{\sigma}^{2}=\sum_{i=1}^{n}\sigma_{i}^{2}/n$. Then there exists a solution \(\hbdbe\) to \eqref{eq-p-03} such that,  with probability at least $1-\varepsilon$,
\begin{align}\label{eq-thm-nonasymtotic-thm}
  \|\boldsymbol{\hat{\beta}}-\bdbe^{*}\|_{2}\leq \frac{L_{n}}{c_{l}}\Big(\alpha_{n}\bar{\sigma}^{2}+\frac{2\log(\varepsilon^{-1})}{n\alpha_{n}}\Big)\cdot(1+\Delta)^{-1}:=\beta_{0}.
\end{align}
\end{theo}
Theorem \ref{thm-nonasymtotic-theory} immediately implies the following corollary, which claims that the estimator $\hbdbe$ is consistent. 
\begin{coro}
  Let $\alpha_{n}=\sqrt{2\log(\varepsilon^{-1})/(n\bar{\sigma}^{2})}$, for $n\geq 4c_{u}L_{n}^{2}\log(\varepsilon^{-1})/c_{l}^{2}$, then with probability at least $1-\varepsilon$,
  \begin{align*}
    \|\boldsymbol{\hat{\beta}}-\bdbe^{*}\|_{2}\leq \frac{2L_{n}}{c_{l}}\sqrt{\frac{2\bar{\sigma}^{2}\log(\varepsilon^{-1})}{n}}\cdot(1+\Delta_{1})^{-1},
  \end{align*}
  where $\Delta_{1}^{2}=1-4c_{u}L_{n}^{2}\log(\varepsilon^{-1})/(nc_{l}^{2}).$
\end{coro}
To obtain the Berry--Esseen bounds for $\hbdbe$, we further require the following assumptions.
\begin{assu}\label{assu-BE-low-demension} There exist two positive constants $K_{0}$ and $K_{1}$ such that $|\dot{\varphi}(x)|\leq K_{0}$ for any $x\in \R$  and $|\dot{\varphi}(x)-\dot{\varphi}(y)|\leq K_{1}|x-y|$ for any $x,y\in\R$.
\end{assu} 
\begin{theo}\label{thm-low-BE-bound} Suppose that $\varepsilon_{1},\ldots, \varepsilon_{n}$ are identically distributed.  Under Assumptions \ref{assu-mean-CLT},\ref{assu-low-dimension-gram-matrix} and \ref{assu-BE-low-demension},  if $\e|\varepsilon_{1}|^{3}<\infty$, then 
\begin{align}\label{eq-thm-low-BE-bound-01}
  &\sup_{A\in \mathcal{A}}\Big|\p\Big(\frac{\sqrt{n}\bdS_{n}^{1/2}(\hbdbe-\bdbe^{*}-\boldsymbol{\delta}_{n})}{ \tilde{\sigma}}\in A \Big{|} \bdx\Big)-\p(Z\in A)\Big|\leq CpL_{n}^{3}n^{-1/2}\frac{\e|\varepsilon_{1}|^{3}}{\sigma^{3}},
\end{align} 
where $\bdx=\{\bdx_{1},\cdots, \bdx_{n}\}$, $\sigma^2=\mathbb{E}[\varepsilon_1^2]$,
\begin{align*}
 \tilde{\sigma}^{2}=\alpha_{n}^{-2}\operatorname{Var}(\varphi[\alpha_{n}\varepsilon_{1}]) ,\quad\boldsymbol{\delta}_{n}=\frac{\bdS_{n}^{-1}}{n\alpha_{n}}\sum_{i=1}^{n}\bdx_{i}\e\varphi[\alpha_{n}\varepsilon_{i}],
\end{align*}
 $\mathcal{A}$ is the collection of all convex sets in $\mathbb{R}^p$, $Z\sim N(0,I_p)$ and $C$ is a positive constant not depending on $n$ and $p$.
\end{theo}

\section{Proofs of main results for mean estimation}\label{section:proofs-of-mean-estimation}
In this section, we provide the proofs of  main results for mean estimation.
Throughout Sections \ref{section:proofs-of-mean-estimation} and \ref{section-proof-of-catoni-type},
$C, C_{1}, C_{2}, \ldots$ and $c_{0}, c_{1}, \ldots$ denote positive constants that may change from line to line.
\subsection{Some preliminary lemmas}
 
In this subsection, we present several preliminary lemmas. We begin by introducing some properties of $\varphi(\cdot)$, which will be used repeatedly in the proofs of the main results.
\begin{lem}\label{lem-1}
  Let $\varphi(\cdot)$ be a function that satisfies condition \ref{assu-mean-clt-a}, then
   for any $x_{1}, x_{2}\in \R$, we have
   \begin{align}\label{eq-lem1-02}
    &|\varphi(x_{1})-\varphi(x_{2})-(x_{1}-x_{2})|\nonumber\\
    &\quad\leq \log\big(1+x_{1}^{2}x_{2}^{2}/2\big)+|x_{1}-x_{2}|\Big(\frac{
    x_{1}^{2}x_{2}^{2}/2}{1+x_{1}^{2}x_{2}^{2}/2}+|x_{1}-x_{2}|\Big).
   \end{align}

\end{lem}
\begin{proof}
  By (\ref{eq-t-04}) and the basic inequality $2ab\leq a^{2}+b^{2}$ for any $a,b\in \R$, we have for any $x_{1}, x_{2}\in \R$,
  \begin{align}\label{eq-pr-lem1-07}
    \varphi(x_{1})-\varphi(x_{2})&\leq \log(1+x_{1}+x_{1}^{2}/2)+\log(1-x_{2}+x_{2}^{2}/2)\nonumber\\
    &=\log\big(1+x_{1}-x_{2}+(x_{1}-x_{2})^{2}/2+x_{1}x_{2}(x_{2}-x_{1})/2+x_{1}^{2}x_{2}^{2}/4\big)\nonumber\\
    &\leq \log \big(1+ x_{1}-x_{2}+3(x_{1}-x_{2})^{2}/4+x_{1}^{2}x_{2}^{2}/2\big).
  \end{align}
  Noting that for any $x>-1$ and $x+y>-1$, we have $1+x+y=(1+x)\cdot[1+y/(1+x)]\leq (1+x)\cdot \exp[y/(1+x)]$, which further implies that
  \begin{align}\label{eq-pr-lem1-08}
    \log(1+x+y)\leq \log(1+x)+\frac{y}{1+x}.
  \end{align}
  Combining (\ref{eq-pr-lem1-07}) and (\ref{eq-pr-lem1-08}) yields that 
  \begin{align}\label{eq-pr-lem1-09}
    &\varphi(x_{1})-\varphi(x_{2})-(x_{1}-x_{2})\nonumber\\
    &\leq \log (1+x_{1}^{2}x_{2}^{2}/2)-(x_{1}-x_{2})\frac{x_{1}^{2}x_{2}^{2}/2}{1+x_{1}^{2}x_{2}^{2}/2}+\frac{3(x_{1}-x_{2})^{2}}{4(1+x_{1}^{2}x_{2}^{2}/2)}\nonumber\\
    &\leq \log (1+x_{1}^{2}x_{2}^{2}/2)+|x_{1}-x_{2}|\frac{x_{1}^{2}x_{2}^{2}/2}{1+x_{1}^{2}x_{2}^{2}/2}+(x_{1}-x_{2})^{2}.
  \end{align}
  By (\ref{eq-pr-lem1-09}) and the symmetry, we obtain (\ref{eq-lem1-02}).
\end{proof}
The following Lemma provides an upper bound for $|u_{n}-u|.$ 
\begin{lem}\label{lem-2}If the condition \ref{assu-mean-clt-a} is satisfied, then 
\begin{align*}
  |u_{n}-u|\leq \frac{\alpha_{n}\sigma^{2}}{\sqrt{1-\alpha_{n}^{2}\sigma^{2}}}.
\end{align*}
\end{lem}
\begin{proof}
 Note that $f(x)=\e \{\varphi[\alpha_{n}(X_{1}-x)]\}$ is non-increasing for $x$, to bound $u_{n}$, we only need to find $u_{+}$, $u_{-}$ such that
  \begin{align*}
    u_{-}<u_{+},\quad f(u_{+})<0\quad\text{and}\quad   f(u_{-})>0.
  \end{align*}   By Jensen's inequality and (\ref{eq-t-04}), we have 
  \begin{align}\label{eq-pp-05}
  f(x)&\leq \log\Big(1+\alpha_{n}\e(X_{1}-x)+\frac{\alpha_{n}^{2}}{2}\e(X_{1}-x)^{2}  \Big)\nonumber\\
   &\leq \log\Big(1+\alpha_{n}(u-x)+\frac{\alpha_{n}^{2}}{2}(\sigma^{2}+(u-x)^{2})  \Big):=f_{+}(x).
  \end{align}
  Similarly, we have
  \begin{align}\label{eq-pp-06}
    f(x)&\geq -\log\Big(1-\alpha_{n}(u-x)+\frac{\alpha_{n}^{2}}{2}(\sigma^{2}+(u-x)^{2})  \Big):=f_{-}(x).
  \end{align}
  By (\ref{eq-pp-05}) and (\ref{eq-pp-06}),  if we can find $u_{+}$, $u_{-}$ such that 
  \[
u_{-}<u_{+},\quad f_{+}(u_{+})< 0\quad\text{and}\quad   f_{-}(u_{-})> 0
  \]
  then $u_{-}-u\leq u_{n}-u\leq u_{+}-u.$ By the definition of $f_{+}(x)$, if $x\leq u$, then $f_{+}(x)>0$. Now, set 
  $u_{+}=u+\eta_{+}$, then
  \begin{align*}
    f_{+}(u_{+})=\log\Big(1-\alpha_{n}\eta_{+}+\frac{\alpha_{n}^{2}}{2}(\sigma^{2}+\eta_{+}^{2})  \Big).
  \end{align*} 
 In order to ensure $f_{+}(u_{+})< 0$, we only need to find  $\eta_{+}$ such that 
\begin{align}\label{eq-pp-07}
 0<1-\alpha_{n}\eta_{+}+\frac{\alpha_{n}^{2}}{2}(\sigma^{2}+\eta_{+}^{2})<1.
\end{align}
Solving the quadric equation
\begin{align*}
    -\alpha_{n}\eta_{0}+\frac{\alpha_{n}^{2}}{2}(\sigma^{2}+\eta_{0}^{2})=0
\end{align*}
shows that its smaller root $\eta_0$ is given by
 \begin{align*}
 \eta_{0}=\frac{\alpha_{n}\sigma^{2}}{1+\sqrt{1-\alpha_{n}^{2}\sigma^{2}}}.
\end{align*}
Consequently, there must exists a 
$\eta_{+}$ within the interval  $ \Big(\frac{\alpha_{n}\sigma^{2}}{1+\sqrt{1-\alpha_{n}^{2}\sigma^{2}}}, \frac{\alpha_{n}\sigma^{2}}{\sqrt{1-\alpha_{n}^{2}\sigma^{2}}}\Big)$ that satisfies (\ref{eq-pp-07}) and hence
\begin{align*}
    u_n-u\le \eta_{+} \le \frac{\alpha_{n}\sigma^{2}}{\sqrt{1-\alpha_{n}^{2}\sigma^{2}}}.
\end{align*}
 Meanwhile, the other half of the inequality follows by performing the same analysis for $f_-$, and thus the lemma is proved.
\end{proof}
The lemma stated below, as cited from \cite{shao-and-dennis2024}, provides
 an exponential lower tail bound for $U$-statistics with non-negative kernels. 
\begin{lem}[\cite{shao-and-dennis2024}]\label{lem-exp-for-sigma} For a sequence of i.i.d. random variables \newline $\{X_{i}\}_{i=1}^{n}$, let 
$U_n=\binom{n}{m}^{-1} \sum_{1 \leq i_1<\cdots<i_m \leq n} h\left(X_{i_1}, \ldots, X_{i_m}\right)$
be a $U$-statistic with degree $m$, where $h: \mathcal{X}^m \longrightarrow \mathbb{R}^{+}$ can only take non-negative values. Assume that 
$$\mathbb{E}[h^{p}]=\mathbb{E}\left[h^p\left(X_1, \ldots, X_m\right)\right]<\infty$$ 
for some $p \in(1,2]$, then for $0<x \leq \mathbb{E}[h]$,
$$
\p\left(U_n \leq x\right) \leq \exp \Big(\frac{-\lfloor n / m \rfloor(p-1)(\mathbb{E}[h]-x)^{p /(p-1)}}{p\left(\mathbb{E}\left[h^p\right]\right)^{1 /(p-1)}}\Big),
$$
where $\lfloor n / m \rfloor$ is defined as the greatest integer less than $n / m$.
\end{lem}
In particular,  when 
$m=1$ and $h(x)=(x-\e X_{1})^{2}$, applying Lemma \ref{lem-exp-for-sigma} with $p=1+\delta/2$ yields 
\begin{align}\label{eq-pq-05}
    \p\Big(n^{-1}\sum_{i=1}^{n}(X_{i}-\e X_{i})^{2}\leq \sigma^{2}/4\Big)\leq \exp\Big(-nc_{\delta} d_{2+\delta}^{(4+2\delta)/\delta} \Big),
\end{align}
where $\sigma^{2}=\e(X_{1}-u)^{2}$, $c_{\delta}=\frac{\delta}{2+\delta}\Big(\frac{3}{4}\Big)^{(2+\delta)/\delta}$ and $d_{k}=\sigma/(\e|X_{1}-u|^{k})^{1/k}$.\\[2pt]
\subsection{Proofs Theorems \ref{thm-mean-CLT} and \ref{thm-mean-md}}
In this subsection, we will sequentially prove Theorems \ref{thm-mean-CLT} and \ref{thm-mean-md}.
 \begin{proof}[Proof of Theorem \ref{thm-mean-CLT}] 
We  divide the analysis into two parts: $|z|\leq \sqrt{\ln n}$ and $|z|\geq \sqrt{\ln n}.$ For $|z|\leq \sqrt{\ln n}$,
  since $\varphi(\cdot)$ is non-decreasing and $\hat{\theta}$ is the solution of (\ref{eq-t-02}), for any  $t\in \R$, 
  \begin{align}\label{eq-ww-01}
  \p\Big(\sum_{i=1}^{n}\varphi\big(\alpha_{n}(X_{i}-t)\big)< 0\Big) \leq  \p\big(\hat{\theta}\leq t \big)\leq \p\Big(\sum_{i=1}^{n}\varphi\big(\alpha_{n}(X_{i}-t)\big)\leq 0\Big).
  \end{align}
 For any  $z\in \R$, define $\delta_{n}=n^{-1/2}z\sigma+u_{n}-u.$ It then follows from (\ref{eq-ww-01}) that 
  \begin{align}\label{eq-ww-02}
    \big|\p\big(\sqrt{n}(\hat{\theta}-u_{n})/\sigma\leq z\big)-\Phi(z)\big|\leq \max\big\{T_{1}(z),T_{2}(z)\big\},
  \end{align}
  where 
  \begin{equation}\label{eq-ww-2.1}
    \begin{aligned}
    T_{1}(z)&=\Big|\p\Big(\sum_{i=1}^{n}\varphi\big(\alpha_{n}(X_{i}-u-\delta_{n})\big)\leq 0\Big)-\Phi(z)\Big|,\\
     T_{2}(z)&=\Big|\p\Big(\sum_{i=1}^{n}\varphi\big(\alpha_{n}(X_{i}-u-\delta_{n})\big)< 0\Big)-\Phi(z)\Big|.
\end{aligned}
  \end{equation}
For any $1\leq i\leq n$, let
  \begin{align}\label{eq-ww-2.2}
\bar{X}_{i}=(X_{i}-u)\mathbf{1}(|X_{i}-u|\leq \sqrt{n}\sigma)\quad \text{and}\quad Y_{ni}(z)=\varphi\big(\alpha_{n}(\bar{X}_{i}-\delta_{n})\big).
\end{align}
With the above notation and by applying Markov's inequality, we have
  \begin{align}\label{eq-ww-03}
    T_{1}(z)&\leq \big|\p\big(\sum_{i=1}^{n}Y_{ni}(z)\leq 0\big)-\Phi(z)\big|+\sum_{i=1}^{n}\p(|X_{i}-u|\geq \sqrt{n}\sigma)\nonumber\\
    &\leq \big|\p(W_{n}(z)\leq z_{n})-\Phi(z)\big|+\beta_{2},
  \end{align}
  where 
  \begin{align*}
  \sigma_{z}^{2}&=\e[Y_{ni}(z)-\e Y_{ni}(z)]^{2},\quad
   z_{n}=-\sqrt{n}\e Y_{n1}(z)/\sigma_{z}, \\
   W_{n}(z)&=\frac{1}{\sqrt{n}}\sum_{i=1}^{n}[Y_{ni}(z)-\e Y_{ni}(z)]/\sigma_{z}. 
  \end{align*}
  With a similar argument as that leading to (\ref{eq-ww-03}), we have 
    \begin{align}\label{eq-ww-04}
    T_{2}(z)\leq \big|\p(W_{n}(z)< z_{n})-\Phi(z)\big|+\beta_{2}.
  \end{align}
  Note that $\p(W_{n}(z)\leq  x-\beta_{2})\leq \p(W_{n}(z)< x)\leq \p(W_{n}(z)\leq  x)$  and $\Phi(x)-\Phi(y)\le |x-y|$ for any $x,y\in \R$. Hence
  \begin{align}\label{eq-ww-05}
    |\p(W_{n}(z)< z_{n})-\Phi(z_{n})|
    &\le \sup_{x} |\p(W_{n}(z)\leq x )-\Phi(x)|+ \sup_{x} \left|\Phi(x)-\Phi(x-\beta_{2})\right|\nonumber\\
    &\leq \sup_{x} |\p(W_{n}(z)\leq x )-\Phi(x)|+\beta_{2}.
  \end{align}
  Combining (\ref{eq-ww-02})--(\ref{eq-ww-05}) yields that 
  \begin{align}\label{eq-ww-06}
   &\sup_{|z|\leq \sqrt{\ln n}}\big|\p\big(\sqrt{n}(\hat{\theta}-u_{n})/\sigma\leq z\big)-\Phi(z)\big|\nonumber\\
   &\quad\leq \sup_{|z|\leq \sqrt{\ln n}}\sup_{x} |\p(W_{n}(z)\leq x )-\Phi(x)|+\sup_{|z|\leq \sqrt{\ln n}}|\Phi(z_{n})-\Phi(z)|+C\beta_{2}
  \end{align}
  We next state two claims, which will be proved below.
  \begin{itemize}
      \item [(A)] For all sufficiently large $n$, 
    there exists $C_{2}>0$ such that
\begin{align}\label{eq-ww-34.1}
  |z_{n}-z|&\leq C_{2}(1+|z|)(\beta_{2}+\beta_{3})+C_{2}(1+z^{2})n^{-1/2}.
\end{align}
    \item [(B)] For all sufficiently large $n$, there exists $C_{3}>0$ such that
     \begin{align}\label{eq-ww-22.1}
     \sup_{|z|\leq \sqrt{\ln n}}\sup_{x}\big|\p(W_{n}(z)\leq x )-\Phi(x)\big|\leq C_{3}\beta_{2}+C_{3}\beta_{3}.
    \end{align}
  \end{itemize}
  For the second term of (\ref{eq-ww-06}), by mean value theorem, 
  \begin{align}\label{eq-ww-25}
    |\Phi(z_{n})-\Phi(z)|\leq |z_{n}-z|e^{-z^{2}/2}\cdot e^{|z|\cdot |z_{n}-z|}.
  \end{align}
  Since $\big|\p\big(\sqrt{n}(\hat{\theta}-u_{n})/\sigma\leq z\big)-\Phi(z)\big|\leq 1$, we may assume that $C_{2}(\beta_{2}+\beta_{3})\leq 1/4$, otherwise (\ref{eq-thm-mean-CLT}) holds with $C=4C_{2}$. By  (\ref{eq-ww-34.1}) and (\ref{eq-ww-25}), we have 
\begin{align}\label{eq-ww-36}
\begin{aligned}
      \sup_{|z|\leq \sqrt{\ln n}}|\Phi(z_{n})-\Phi(z)|&\leq  C(\beta_{2}+\beta_{3})\cdot (1+z^{2})\exp(-z^{2}/4+|z|/4)\\
  &\leq C(\beta_{2}+\beta_{3}).
\end{aligned}

\end{align}
Combining (\ref{eq-ww-06}), (\ref{eq-ww-22.1}) and (\ref{eq-ww-36}) yields
\begin{align}\label{eq-ww-37}
  \sup_{|z|\leq \sqrt{\ln n}}\big|\p\big(\sqrt{n}(\hat{\theta}-u_{n})/\sigma\leq z\big)-\Phi(z)\big|\leq C(\beta_{2}+\beta_{3}).
\end{align}
For $z\leq -\sqrt{\ln n}$, it follows from (\ref{eq-ww-37}) that 
  \begin{align}\label{eq-ww-23}
    &\big|\p\big(\sqrt{n}(\hat{\theta}-u_{n})/\sigma\leq z\big)-\Phi(z)\big|\nonumber\\
    &\quad\leq \p\big(\sqrt{n}(\hat{\theta}-u_{n})/\sigma\leq -\sqrt{\ln n}\big)+\Phi(-\sqrt{\ln n})\nonumber\\
    &\quad\leq \big|\p\big(\sqrt{n}(\hat{\theta}-u_{n})/\sigma\leq -\sqrt{\ln n}\big)-\Phi(-\sqrt{\ln n})\big|+2\Phi(-\sqrt{\ln n})\nonumber\\
    &\quad\leq C\beta_{2}+C\beta_{3}+Cn^{-1/2},
  \end{align}
  where the last inequality uses the fact that $\Phi(-\sqrt{\ln n})\leq \exp(-0.5\ln n)\leq n^{-1/2}.$\\
  For $z\geq \sqrt{\ln n}$, by (\ref{eq-ww-37}) again, we have
   \begin{align}\label{eq-ww-24}
    &\big|\p\big(\sqrt{n}(\hat{\theta}-u_{n})/\sigma\leq z\big)-\Phi(z)\big|\nonumber\\
    &\quad=\big|\p\big(\sqrt{n}(\hat{\theta}-u_{n})/\sigma> z\big)-[1-\Phi(z)]\big|\nonumber\\
    &\quad\leq \p\big(\sqrt{n}(\hat{\theta}-u_{n})/\sigma>\sqrt{\ln n}\big)+1-\Phi(\sqrt{\ln n})\nonumber\\
     &\quad\leq \big|\p\big(\sqrt{n}(\hat{\theta}-u_{n})/\sigma>\sqrt{\ln n}\big)-[1-\Phi(\sqrt{\ln n})]\big|+2[1-\Phi(\sqrt{\ln n})]\nonumber\\
    &\quad\leq C\beta_{2}+C\beta_{3}+Cn^{-1/2}.
  \end{align} 
Moreover, by the definitions of $\beta_{2}$ and $\beta_{3}$ and H\"{o}lder inequality, we have $\sqrt{n}\beta_{3}\geq (1-\beta_{2})^{3/2}$, which implies
  \begin{align}\label{eq-ww-24.1}
     & \sqrt{n}(\beta_{2}+\beta_{3})\geq \beta_{2}+(1-\beta_{2})^{3/2}\geq \min_{0\leq s\leq 1}\{s+(1-s)^{3/2}\}\geq 23/27.
  \end{align}
  Now, (\ref{eq-thm-mean-CLT}) follows form (\ref{eq-ww-37}), (\ref{eq-ww-23}),  (\ref{eq-ww-24}) and \eqref{eq-ww-24.1}. What remains is to prove (\ref{eq-ww-34.1}) and (\ref{eq-ww-22.1}). \par\medskip

\noindent
{\it Proof of (\ref{eq-ww-34.1}).}
   For the sake of convenience in writing, define 
  \begin{align*}
  L_{n}(x)=\e \big\{\varphi\big(\alpha_{n}(\bar{X}_{1}-x)\big)\big\},\quad M_{n}(x)=\e \big\{\varphi^{2}\big(\alpha_{n}(\bar{X}_{1}-x)\big)\big\}.
  \end{align*}
 Recall that $z_{n}=-\sqrt{n}\e  \big\{\varphi[\alpha_{n}(\bar{X}_{1}-\delta_{n})]\big\}/\sigma_{z}$ and $\delta_{n}=n^{-1/2}z\sigma+u_{n}-u$. Then
  \begin{align}\label{eq-ww-27}
    \begin{split}
        |z-z_{n}|&=|z+\sqrt{n}\e  \big\{\varphi[\alpha_{n}(\bar{X}_{1}-\delta_{n})]\big\}/\sigma_{z}|
    \\&=\Big|z-\frac{\sqrt{n}\alpha_{n}}{\sigma_{z}}[ \delta_{n} -(u_{n}-u)]+\frac{\sqrt{n}}{\sigma_{z}}\big[L_{n}(\delta_{n})+\alpha_{n}[ \delta_{n} -(u_{n}-u)]\big]\Big|\\
    &\leq z\cdot\Big|\frac{\alpha_{n} \sigma}{\sigma_{z}}-1\Big|+\frac{\sqrt{n}}{\sigma_{z}}\Big| L_{n}( \delta_{n} )-L_{n}(u_{n}-u)+\alpha_{n}[ \delta_{n} -(u_{n}-u)]\Big|\\
    &\quad +\frac{\sqrt{n}|L_{n}(u_{n}-u)|}{\sigma_{z}}\\
    &:=z\cdot Q_{n,1} +Q_{n,2}+Q_{n,3}.
    \end{split}
  \end{align}
In what follows, we shall  provide  the upper bounds of  $Q_{n,1}, Q_{n,2},Q_{n,3}$ respectively.\par
\medskip
\noindent
{\it (i) Upper bound of $Q_{n,1}.$}
  For $Q_{n,1}$, note that 
  \begin{align*}
  Q_{n,1}\leq \Big|\frac{\alpha_{n} \sigma}{\sigma_{z}}-1\Big|\cdot \Big|\frac{\alpha_{n} \sigma}{\sigma_{z}}+1\Big|=\frac{\alpha_{n}^{2}}{\sigma_{z}^{2}}\cdot \big|\sigma^{2}-\sigma_{z}^{2}/\alpha_{n}^{2}\big|.
  \end{align*}
  So to bound $Q_{n,1}$, it suffices to provide an upper bound for $\big|\sigma^{2}-\sigma_{z}^{2}/\alpha_{n}^{2}\big|.$ By the triangle inequality, we have
  \begin{align}\label{eq-ww-27.1}
   \big|\sigma^{2}-\sigma_{z}^{2}/\alpha_{n}^{2}\big|\leq \frac{|\alpha_{n}^{2}\sigma^{2}-M_{n}(0)|}{\alpha_{n}^{2}}+
  \frac{|M_{n}(0)-\sigma_{z}^{2}|}{\alpha_{n}^{2}}.  
  \end{align}
  For the first term of (\ref{eq-ww-27.1}), by definition of $\bar{X}_1$ and  inequalities $|\varphi(x)|\leq |x|$ and $|\varphi(x)-x|\leq x^{2}$, we have
  \begin{align}\label{eq-ww-28}
    \alpha_{n}^{-2}|\alpha_{n}^{2}\sigma^{2}-M_{n}(0)|&\leq \alpha_{n}^{-2}\e \big|\varphi^{2}\big[\alpha_{n} \bar{X}_{1}\big]-\alpha_{n}^{2}(X_{1}-u)^{2}\big|\nonumber\\
    &\leq \alpha_{n}^{-2}\e \big|\varphi^{2}\big[\alpha_{n} \bar{X}_{1}\big]-\alpha_{n}^{2}\bar{X}_{1}^{2}\big|+\e(X_{1}-u)^{2}\mathbf{1}(|X-u|\geq \sqrt{n}\sigma)\nonumber\\
    &\leq 2\alpha_{n}^{-1}\e\big\{|\bar{X}_{1}|\cdot \big|\varphi[\alpha_{n}\bar{X}_{1}]-\alpha_{n}\bar{X}_{1}\big|\big\}+\sigma^{2}\beta_{2}\nonumber\\
    &\leq 2\alpha_{n}\e|\bar{X}_{1}|^{3}+\sigma^{2}\beta_{2}.
  \end{align} 
  For the second term of (\ref{eq-ww-27.1}),  by Lemma \ref{lem-1}, 
  \begin{align}\label{eq-ww-13}
   |\varphi(x_{1})-\varphi(x_{2})|\leq \log \big(1+x_{1}^{2}x_{2}^{2}/2\big)+2|x_{1}-x_{2}|+(x_{1}-x_{2})^{2}.
  \end{align} Recall that
  \begin{align*}
      \sigma_{z}^{2}=\operatorname{Var}\left(Y_{n1}(z)\right)=\e \big\{\varphi^{2}[\alpha_{n}(\bar{X}_{1}-\delta_{n})]\big\}-\left(\e\{ \varphi[\alpha_{n}(\bar{X}_{1}-\delta_{n})]\}\right)^{2}.
  \end{align*}
So, by applying (\ref{eq-ww-13}) with \( x_1 = \alpha_{n} \bar{X}_1 \) and \( x_2 = \alpha_{n}(\bar{X}_1 - \delta_n) \), and using the inequlity \( |\varphi(x)| \leq |x| \), we have
  \begin{align}\label{eq-pr-23}
 \frac{|M_{n}(0)-\sigma_{z}^{2}|}{\alpha_{n}^{2}}
    &\leq \alpha_{n}^{-2}\Big|\e \big\{\varphi^{2}[\alpha_{n}\bar{X}_{1}]\big\}-\e \big\{\varphi^{2}[\alpha_{n}(\bar{X}_{1}-\delta_{n})]\big\}\Big|+\big[\alpha_{n}^{-1}L_{n}(\delta_{n})\big]^{2}\nonumber\\
    &\leq \alpha_{n}^{-1}\e\Big\{\big|\varphi[\alpha_{n}\bar{X}_{1}]- \varphi[\alpha_{n}(\bar{X}_{1}-\delta_{n})]\big|\cdot (|\bar{X}_{1}|+|\bar{X}_{1}-\delta_{n}|)\Big\} \nonumber \\
    &\quad +\big[\alpha_{n}^{-1}L_{n}(\delta_{n})\big]^{2}\nonumber\\
    &\leq H_{1}+H_{2}+H_{3},
  \end{align}
  where 
  \begin{align*}
    H_{1}&=\alpha_{n}^{-1}\e\Big\{\log\big(1+\alpha_{n}^{4}\bar{X}_{1}^{2}(\bar{X}_{1}-\delta_{n})^{2}/2\big)\cdot (|\bar{X}_{1}|+|\bar{X}_{1}-\delta_{n}|)\Big\},\\
    H_{2}&=\alpha_{n}^{-1}\e\Big\{\big(2\alpha_{n}|\delta_{n}|+\alpha_{n}^{2}|\delta_{n}|^{2}\big)\cdot (|\bar{X}_{1}|+|\bar{X}_{1}-\delta_{n}|)\Big\},\nonumber\\
    H_{3}&=\big[\alpha_{n}^{-1}L_{n}(\delta_{n})\big]^{2}.
  \end{align*}
   Since $\limsup_{n\to \infty} \sqrt{n}a_{n}\leq C_{0}$, there exists a constant $C'_{0}$ such that
$a_{n}\leq C'_{0}n^{-1/2}$ for all sufficiently large $n$. In addition, by Lemma \ref{lem-2}, we have 
for large $n$,
\begin{align}\label{eq-ww-15}
  |\delta_{n}|\leq n^{-1/2}z\sigma+|u_{n}-u|\leq n^{-1/2}z\sigma+Ca_{n}\sigma\leq C(1+|z|)n^{-1/2}\sigma,
\end{align}
and hence 
\begin{align}\label{eq-ww-16}
  |\delta_{n}|\leq  C\sigma\quad \text{for all sufficiently large }n \text{ and } |z|\leq \sqrt{\ln n}.
\end{align}
  For $H_{1}$, by the inequality $\log(1+x^{2}/2)\leq |x|,\ \forall x\in\R$ and (\ref{eq-ww-16}), we have 
  \begin{align}\label{eq-n1-01}
    H_{1}&\leq \alpha_{n}\e\big\{|\bar{X}_{1}|\cdot|\bar{X}_{1}-\delta_{n}|\cdot (|\bar{X}_{1}|+|\bar{X}_{1}-\delta_{n}|)\big\}\nonumber\\
    &\leq Cn^{-1/2}\sigma^{-1} \e\big\{|\bar{X}_{1}|\cdot(|\bar{X}_{1}|+\sigma)\cdot (2|\bar{X}_{1}|+\sigma)\big\}\nonumber\\
    &\leq Cn^{-1/2}\big(\sigma\e|\bar{X}_{1}|+\e|\bar{X}_{1}|^{2}+\sigma^{-1}\e|\bar{X}_{1}|^{3}\big)\nonumber\\
    &\leq Cn^{-1/2}\sigma^{2}+C\sigma^{2}\beta_{3}.
  \end{align}
  For $H_{2}$, it follows from  (\ref{eq-ww-15}) and (\ref{eq-ww-16}) again that
 \begin{align}\label{eq-n1-02}
   H_{2}&\leq C\e\big\{\big(|\delta_{n}|+\alpha_{n}\sigma^{2}\big)\cdot (2|\bar{X}_{1}|+\sigma)\big\}\leq Cn^{-1/2}(1+|z|)\sigma^{2}.
  \end{align}
  For $H_{3}$, by the definition of $\bar{X}_{1}$ in \eqref{eq-ww-2.2}, it is clear that $\e[\bar{X}_{1}^{2}]\le \sigma^{2}$ and
  \begin{align*}
      |\e[\bar{X}_1]|=|-\e(X_1-u)\mathbf{1}(|X_1-u|\geq \sqrt{n}\sigma)|\le \frac{1}{\sqrt{n}\sigma} \e[|X_1-u|^{2}] =\frac{\sigma}{\sqrt{n}}.
  \end{align*}
  So it follows from (\ref{eq-ww-15}) and (\ref{eq-ww-16}) that 
  \begin{align}\label{eq-ww-29.1}
      \alpha_{n}^{-1}|\e\varphi\big[\alpha_{n}(\bar{X}_{1}-\delta_{n})\big]|&\leq \alpha_{n}^{-1}\e\big|\varphi\big[\alpha_{n}(\bar{X}_{1}-\delta_{n})\big]-\big[\alpha_{n}(\bar{X}_{1}-\delta_{n})\big]\big|+|\e\bar{X}_{1}-\delta_{n}|\nonumber\\
      &\leq \alpha_{n}\e(\bar{X}_{1}-\delta_{n})^{2}+|\e\bar{X}_{1}-\delta_{n}|\nonumber\\
      &\leq 2\alpha_{n}\e\bar{X}_{1}^{2}+(1+2\alpha_{n}\delta_{n})\delta_{n}+|\e\bar{X}_{1}|\nonumber\\
      &\leq Cn^{-1/2}(1+|z|)\sigma
  \end{align}
 Noting that $|z| \leq \sqrt{\ln n}$, so $n^{-1/2} |z| \leq  \sqrt{\ln n /n} \leq 1$. Combining this with (\ref{eq-n1-01}), (\ref{eq-n1-02}), (\ref{eq-ww-29.1}) yields that 
 \begin{align}\label{eq-ww-29.2}
 \alpha_{n}^{-2}|M_{n}(0)-\sigma_{z}^{2}|\leq Cn^{-1/2}(1+|z|)\sigma^{2}+C\sigma^{2}(\beta_{2}+\beta_{3}).
 \end{align}
 Substituting 
(\ref{eq-ww-28}) and 
(\ref{eq-ww-29.2}) into (\ref{eq-ww-27.1}), we obtain
  \begin{align}\label{eq-ww-29.3}
\big|\sigma^{2}-\sigma_{z}^{2}/\alpha_{n}^{2}\big|&\leq Cn^{-1/2}(1+|z|)\sigma^{2}+C\sigma^{2}(\beta_{2}+\beta_{3}).
\end{align}
By (\ref{eq-nn-0}) and (\ref{eq-ww-29.3}), for $|z|\leq \sqrt{\ln n}$, we  have $\big|\sigma^{2}-\sigma_{z}^{2}/\alpha_{n}^{2}\big|\to 0$ as $n\to \infty$. This implies that 
 \begin{align}\label{eq-ww-21}
 \sigma_{z}^{2}/\alpha_{n}^{2}\geq \sigma^{2}/2\quad \text{for all sufficiently large } n.
 \end{align}
 Combine (\ref{eq-ww-29.3}) and (\ref{eq-ww-21}) yields
  \begin{align}\label{eq-ww-30}
   Q_{n,1}&\leq C\beta_{2}+C\beta_{3}+Cn^{-1/2}(1+|z|).
  \end{align}

\noindent
{\it (ii) Upper bound of $Q_{n,2}.$} By (\ref{eq-lem1-02}) and the inequalities $\log(1+x^{4}/2)\leq 2|x|^{3}$ and $|x|\leq \frac{\sqrt{2}}{2}(1+\frac{x^{2}}{2})$, we have   
\begin{align}\label{eq-ww-31}
  |\varphi(x_{1})-\varphi(x_{2})-(x_{1}-x_{2})|
    &\leq \log\big(1+x_{1}^{2}x_{2}^{2}/2\big)+|x_{1}-x_{2}|\Big(\frac{
    x_{1}^{2}x_{2}^{2}/2}{1+x_{1}^{2}x_{2}^{2}/2}+|x_{1}-x_{2}|\Big)\nonumber\\
    &\leq 2|x_{1}x_{2}|^{3/2}+0.4|x_{1}x_{2}|\cdot|x_{1}-x_{2}|+|x_{1}-x_{2}|^{2}.
\end{align}
By \eqref{eq-ww-21} and applying (\ref{eq-ww-31}) with $x_{1}=\alpha_{n}(\bar{X}_{1}-\delta_{n})$ and $x_{2}=\alpha_{n}[\bar{X}_{1}-(u_{n}-u)]$, we obtain
\begin{align}\label{eq-ww-32}
 Q_{n,2}&\leq \frac{C\sqrt{n}}{a_{n}}\Big| \varphi[\alpha_{n}(\bar{X}_{1}-\delta_{n})]-\varphi[\alpha_{n}(\bar{X}_{1}-(u_{n}-u))]+\alpha_{n}[\delta_{n}-(u_{n}-u)]\Big|\nonumber\\
 &\leq C\frac{\sqrt{n}}{a_{n}}\cdot \Big( \alpha_{n}^{3}|(\bar{X}_{1}-\delta_{n})\cdot (\bar{X}_{1}-(u_{n}-u))|^{3/2}
 +n^{-1/2}a_{n}^{2}z^{2}\nonumber\\
 &\quad\qquad\qquad\quad+\alpha_{n}^{2}|(\bar{X}_{1}-\delta_{n})\cdot (\bar{X}_{1}-(u_{n}-u))|\cdot n^{-1/2}a_{n}z\Big)\nonumber\\
 &\leq C\sqrt{n}a_{n}^{2}\sigma^{-3}\e|(\bar{X}_{1}-\delta_{n})\cdot (\bar{X}_{1}-(u_{n}-u))|^{3/2}+Ca_{n}z^{2}\nonumber\\
 &\quad +Cz\alpha_{n}^{2}|(\bar{X}_{1}-\delta_{n})\cdot (\bar{X}_{1}-(u_{n}-u))|\nonumber\\
 &\leq Cn^{-1/2}\sigma^{-3}\e( \sigma^{2}+\sigma|\bar{X}_{1}|+|\bar{X}_{1}|^{2})^{3/2}+Cn^{-1/2}z^{2}\nonumber\\
 &\quad +Cn^{-1}z\sigma^{-2}\e(\sigma^{2}+\sigma|\bar{X}_{1}|+|\bar{X}_{1}|^{2})\nonumber\\
 &\leq Cn^{-1/2}(1+z^{2})+Cn^{-1/2}\sigma^{-3}\e|\bar{X}_{1}|^{3}\quad \text{for large }n,
\end{align}
where  the second-to-last inequality follows from the fact $|u_{n}-u|\leq C\sigma$ and $|\delta_{n}|\leq C\sigma$ for large $n$.\par
\medskip
\noindent
{\it (iii) Upper bound of $Q_{n,3}.$} Noting that $\e \big\{\varphi\big(\alpha_{n}(X_{1}-u_{n})\big)\big\}=0$ by the definition of $u_n$, then 
\begin{align*}
 &|L_{n}(u_{n}-u)|\nonumber\\
 &=\e \big\{\varphi\big[\alpha_{n}(\bar{X}_{1}-u_{n}+u)\big]\big\}\nonumber\\
 &=\e \big\{\varphi\big[\alpha_{n}(\bar{X}_{1}-u_{n}+u)\big]\mathbf{1}(|X_{1}-u|> \sqrt{n}\sigma)\big\}\nonumber\\
 &\quad +\e \big\{\varphi\big[\alpha_{n}(\bar{X}_{1}-u_{n}+u)\big]\mathbf{1}(|X_{1}-u|\leq  \sqrt{n}\sigma)\big\}\nonumber\\
 &=\e \big\{\varphi\big[\alpha_{n}(0-u_{n}+u)\big]\mathbf{1}(|X_{1}-u|> \sqrt{n}\sigma)\big\}\nonumber\\
 &\quad +\e \big\{\varphi\big[\alpha_{n}(X_{1}-u_{n})\big]\mathbf{1}(|X_{1}-u|\leq  \sqrt{n}\sigma)\big\}\nonumber\\
 &=\e \big\{\varphi\big[\alpha_{n}(u-u_{n})\big]\mathbf{1}(|X_{1}-u|> \sqrt{n}\sigma)\big\}-\e \big\{\varphi\big[\alpha_{n}(X_{1}-u_{n})\big]\mathbf{1}(|X_{1}-u|>  \sqrt{n}\sigma)\big\}.
\end{align*}
By Markov's inequality and the fact that $|\varphi(x)|\leq x$, $|u_{n}-u|\leq C\sigma$, we have 
\begin{align*}
  &\e \big\{\varphi\big[\alpha_{n}(u-u_{n})\big]\mathbf{1}(|X_{1}-u|> \sqrt{n}\sigma)\big\}\leq Ca_{n}\p(|X_{1}-u|> \sqrt{n}\sigma)\leq Ca_{n}n^{-1}\beta_{2}.
\end{align*}
Similarly, we have 
\begin{align*}
  \e \big\{\varphi\big[\alpha_{n}(X_{1}-u_{n})\big]\mathbf{1}(|X_{1}-u|>  \sqrt{n}\sigma)\big\}\leq Ca_{n}n^{-1}\beta_{2}+Ca_{n}n^{-1/2}\beta_{2}.
\end{align*}
and hence 
\begin{align*}
    |L_{n}(u_{n}-u)|\le Ca_nn^{-1/2}\beta_2.
\end{align*}
Combining this with \eqref{eq-ww-21} gives us
\begin{align}\label{eq-ww-33}
  Q_{n,3}&\leq  C\beta_{2}.
\end{align}
Combining (\ref{eq-ww-27}), (\ref{eq-ww-30}), (\ref{eq-ww-32}) and (\ref{eq-ww-33}) yields that for all sufficiently large $n$, there exists $C_{2}>0$ such that
\begin{align}\label{eq-ww-34}
  |z_{n}-z|&\leq C_{2}(1+|z|)(\beta_{2}+\beta_{3})+C_{2}(1+z^{2})n^{-1/2},
\end{align}
so the proof of (\ref{eq-ww-34.1}) is complete.\par
\medskip
  \noindent
  {\it Proof of (\ref{eq-ww-22.1}).}
     Applying  \citet[Theorem 3.6]{chen2010normal} with $\xi_{i}= (Y_{ni}-\e Y_{ni})/(\sqrt{n}\sigma_{z})$ yields that
  \begin{align}\label{eq-ww-07}
   \sup_{x\in\R}\big|\p(W_{n}(z)\leq x  )-\Phi(x)\big|&\leq C\sum_{i=1}^{n}\e|\xi_{i}|^{3}\leq C n^{-1/2}\sigma_{z}^{-3}\cdot\e|Y_{n1}(z)-\e Y_{n1}(z)|^{3}\nonumber\\
   &\leq  C n^{-1/2}\sigma_{z}^{-3}\cdot\e|Y_{n1}(z)|^{3},
  \end{align}
  where the last inequality follows from the inequality $(a+b)^{p}\leq 2^{p-1}(a^{p}+b^{p})$ for all $p\geq 1$, $a,b\geq 0$. By the definition of $Y_{ni}(z)$, $|\varphi(x)|\leq |x|$ and Lemma \ref{lem-2}, whenever $|z|\le \sqrt{\ln{n}}$, we have
  \begin{align}\label{eq-ww-08}
   \e|Y_{n1}(z)|^{3} &\leq \alpha_{n}^{3}\e|\bar{X}_{1}+n^{-1/2}z\sigma+u_{n}-u|^{3}\nonumber\\
    &\leq C\alpha_{n}^{3}(\e|\bar{X}_{1}|^{3}+(n/\ln n)^{-3/2}\sigma^{3}+|u_{n}-u|^{3})\nonumber\\
    &\leq C\alpha_{n}^{3}(\sqrt{n}\sigma^{3}\beta_{3}+\sigma^{3})\quad \text{for large }n.
  \end{align}
  Thus the proof of \eqref{eq-ww-22.1} is completed by combining \eqref{eq-ww-21}, \eqref{eq-ww-07}, \eqref{eq-ww-08} and \eqref{eq-ww-24.1}.

\end{proof}

\begin{proof}[Proof of Theorem \ref{thm-mean-md}] When $0\leq z\leq 1$, the desired result (\ref{eq-thm-mean-md-1}) follows directly from Theorem \ref{thm-mean-CLT}. Therefore, it suffices to  consider $1\leq z\leq O(n^{1/6}).$ Recall that by (\ref{eq-ww-02}),  
  \begin{align*}
\big|\p\big(\sqrt{n}(\hat{\theta}-u_{n})/\sigma> z\big)-[1-\Phi(z)]\big|&=\big|\p\big(\sqrt{n}(\hat{\theta}-u_{n})/\sigma\leq z\big)-\Phi(z)\big|\nonumber\\
&\leq \max\{T_{1}(z), T_{2}(z)\},
  \end{align*}
  where $T_{1}(z)$ and $T_{2}(z)$ are defined as (\ref{eq-ww-2.1}). So to prove Theorem \ref{thm-mean-md}, it suffices to bound  $T_{1}(z)$ and $T_{2}(z)$.  We will only estimate $T_{2}(z)$, as the argument for $T_{1}(z)$ is similar. For the sake of convenience, define 
  \begin{align*}
   Y'_{ni}(z)&=\varphi\big(\alpha_{n}(X_{i}-u-\delta_{n})\big), \quad (\sigma'_{z})^{2}=\e[Y'_{ni}(z)-\e Y'_{ni}(z)]^{2}, \nonumber\\
   W'_{n}(z)&=\frac{1}{\sqrt{n}}\sum_{i=1}^{n}[Y'_{ni}(z)-\e Y'_{ni}(z)]/\sigma'_{z}\quad  z'_{n}=-\sqrt{n}\e Y'_{n1}(z)/\sigma'_{z}.
  \end{align*}
  We remark that $\sigma'_{z}$ and $z'_{n}$ are obtained by replacing $\bar{X}_{i}$ with $X_{i}-u$ in the definitions of 
$\sigma_{z}$ and $z_{n}$, respectively. Noting that $z\leq O(n^{1/6})$,   following the same arguments used in the proofs of 
(\ref{eq-ww-21}) and (\ref{eq-ww-34.1}), we obtain the following result:   for all sufficiently large $n$,
   \begin{align}\label{eq-we-05}
     a_{n}/\sigma'_{z}\leq 4,\quad |z'_{n}-z|\leq C n^{-1/2}(1+z^{2})\sigma^{-3}\e|X_{1}-u|^{3}:=\tau_{n}. 
   \end{align}
   To avoid repetition, we omit the details here.
   Observe that $\tau_{n}/z\to 0$ as $n\to \infty$, then for all sufficiently large $n$,  $\tau_{n}/z\leq 1$, and hence 
   \begin{align}\label{eq-we-06}
    0\leq z'_{n}\leq 2z\quad \text{for all sufficiently large } n.
   \end{align}
   Let 
  $$
   \xi'_{i}=(Y'_{ni}(z)-\e Y'_{ni}(z))/\sigma'_{z},
 $$
 then $\xi'_{1},\xi'_{2},\cdots,\xi'_{n}$ are i.i.d. random variables with $\e \xi'_{i}=0$ and $\sum_{i=1}^{n}\e (\xi'_{i})^{2}=n$.  Taking $t_{1}=t_{0}\sqrt{\sigma}/2$, then by (\ref{eq-we-05}), $t_{1}(\sigma'_{z})^{-1/2}=t_{0}\alpha_{n}^{-1/2}\sqrt{a_{n}/(4\sigma'_{z})}\leq t_{0}\alpha_{n}^{-1/2}.$ So by inequality $|\varphi(x)|\leq x$, we have   for $1\leq z\leq O(n^{1/6})$, 
 \begin{align*}
   \e\Big\{\exp\Big(t_{1}\sqrt{|\xi'_{i}|}\Big)\Big\}&\leq \exp\big( t_{1}\sqrt{\e  |Y'_{ni}(z)|/\sigma'_{z}}\big)\cdot\e\exp\big(t_{1} \sqrt{Y'_{ni}(z)/\sigma'_{z}} \big)\nonumber\\
   &\leq \exp\big(t_{0}\sqrt{\e  |Y'_{ni}(z)|/\alpha_{n}}\big)\cdot\e\exp\big(t_{0}\sqrt{|Y'_{ni}(z)|/\alpha_{n}}\big)\nonumber\\
   &\leq C\e\Big\{\exp\Big[t_{0}\sqrt{|X_{1}-u|+n^{-1/2}z\sigma+|u_{n}-u|}\Big]\Big\}\nonumber\\
   &\leq C\e\Big\{\exp\Big[t_{0}\sqrt{|X_{1}-u|+C\sigma}\,\Big]\Big\}<\infty,
 \end{align*}
since $\e e^{t_{0}\sqrt{|X_{1}|}}<\infty.$ Then   
 by \citet[Proposition 4.6]{Shaoandchen(2013)}, we have
  \begin{align}\label{eq-we-08}
\Big|\p\Big(\sum_{i=1}^{n}\xi'_{i}/\sqrt{n}\geq z'_{n}\Big)-[1-\Phi(z'_{n})]\Big|\leq Cn^{-1/2}(1+(z'_{n})^{3})(1-\Phi(z'_{n})),
  \end{align}
 for $0\leq z'_{n}\leq t_{1}^{2/3}n^{1/6}$. By (\ref{eq-we-08}) and the definition of $T_{2}(z)$, we have 
  \begin{align}\label{eq-we-09}
  T_{2}(z)&=\Big|\p\Big(\sum_{i=1}^{n}\xi_{i}/\sqrt{n}\geq z'_{n}\Big)-[1-\Phi(z)]\Big|\nonumber\\
  &\leq Cn^{-1/2}(1+(z'_{n})^{3})(1-\Phi(z'_{n}))+ \big|\Phi(z'_{n})-\Phi(z)\big|.
  \end{align}
  For the second term of (\ref{eq-we-09}), since for any $x\in \R,y\geq 1$,
  \begin{align}\label{eq-we-10}
  \big|\Phi(x)-\Phi(y)\big|&\leq  [1-\Phi(y)]\cdot 2\sqrt{2\pi}ye^{y^{2}/2}|\Phi(x)-\Phi(y)|\nonumber\\
  &\leq 2[1-\Phi(y)]\cdot
   y\cdot|y-x|e^{y\cdot |y-x|},
  \end{align}
  where the first inequality follows from the following  well-known inequality
  \begin{align}\label{eq-we-11}
  \frac{ye^{-y^{2}/2}}{(1+y^{2})\sqrt{2\pi}} \leq  1-\Phi(y)\leq \frac{e^{-y^{2}/2}}{y\sqrt{2\pi}}, \quad \text{for all } y>0,
  \end{align}  
  By (\ref{eq-we-05}), we have for $1\leq z\leq 0.5t_{1}^{2/3}n^{1/6}$,
  \begin{align}\label{eq-we-11.1}
  z\cdot|z'_{n}-z|\leq Cn^{-1/2}(1+z^{3})\leq C
  \end{align}
  Combining  (\ref{eq-we-10}) and the two bounds in (\ref{eq-we-11.1}) yields that 
   \begin{align}\label{eq-we-20}
  \big|\Phi(z'_{n})-\Phi(z)\big|&\leq Cn^{-1/2}(1+z^{3})[1-\Phi(z)]\leq C[1-\Phi(z)] 
 \end{align}
  By (\ref{eq-we-06}), (\ref{eq-we-09}) and the two bounds in (\ref{eq-we-20}), we have 
  \[
  T_{2}(z)\leq Cn^{-1/2}(1+z^{3})[1-\Phi(z)].
  \]
This proves (\ref{eq-thm-mean-md-1}).
\end{proof}

\subsection{Proofs of Corollary \ref{coro-mean-MD} and Theorems \ref{thm-self-mean-clt} and \ref{thm-self-mean-md}}
\begin{proof}[Proof of Corollary \ref{coro-mean-MD}] Consider $1\leq z\leq 0.5c_{0}\min\{n^{1/6}, n^{-1/2}a_{n}^{-1}\}$. Let $z'=z-\sqrt{n}(u-u_{n})/\sigma$, then   
\begin{align}\label{eq-wq-01}
|z'-z|\leq \sqrt{n}|u-u_{n}|/\sigma\leq C_{4}\sqrt{n}a_{n}
\end{align}
for some positive constant $C_{4}$. Without loss of generality, we assume that $C_{4}\sqrt{n}a_{n}\leq 1$. Otherwise,  $1-\Phi(z)\geq 1-\Phi(0.5c_{0}n^{-1/2}a_{n}^{-1})\geq 1-\Phi(0.5c_{0}C_{4})>0$, which implies that Corollary \ref{coro-mean-MD} holds trivially. Now, since  $|z'-z|\leq 1\leq z$, it follows that $0\leq z'\leq c_{0}n^{1/6}$ for all sufficiently large $n$. Moreover, by (\ref{eq-we-11}), (\ref{eq-wq-01}), for $1\leq z\leq 0.5c_{0}\min\{n^{1/6}, n^{-1/2}a_{n}^{-1}\}$,
  \begin{align}
  \big|[1-\Phi(z')]-[1-\Phi(z)]\big|&\leq C[1-\Phi(z)]\cdot
   z\cdot|z'-z|e^{z\cdot |z'-z|}\nonumber\\
   &\leq C[1-\Phi(z)] (1+z)\sqrt{n}a_{n} \label{eq-wq-02}\\
   &\leq C[1-\Phi(z)].\label{eq-wq-03}
  \end{align}
  Then Corollary \ref{coro-mean-MD} follows from Theorem \ref{thm-mean-md}, (\ref{eq-wq-02}), (\ref{eq-wq-03}) and the following inequality  
\begin{align*}
 &\big|\p\big(\sqrt{n}(\hat{\theta}-u)/\sigma> z\big)-[1-\Phi(z)]\big|\nonumber\\
  &\quad\leq  \big|\p\big(\sqrt{n}(\hat{\theta}-u_{n})/\sigma> z'\big)-[1-\Phi(z')]\big|+|\Phi(z')-\Phi(z)|.
\end{align*} 
\end{proof}
\begin{proof}[Proof of Theorem \ref{thm-self-mean-clt}]
We only prove that (\ref{eq-thm-self-mean-clt-01}) holds for $|z|\leq \sqrt{\ln n}$.  The proof for  $|z|\geq \sqrt{\ln n}$ can be obtained using an argument similar to that used in the proofs of (\ref{eq-ww-23}) and (\ref{eq-ww-24}).
With a similar argument as that leading to (\ref{eq-ww-02}), we have 
\begin{align}\label{eq-self-mean--01}
\big|\p(\sqrt{n}(\hat{\theta}_{s}-u)/\hat{\sigma}> z)-[1-\Phi(z)]\big|\leq \max\{T_{3}(z), T_{4}(z)\}.
\end{align}
where 
\begin{align*}
  T_{3}(z)&=\Big|\p\Big(\frac{1}{\sqrt{n}a_{n}}\sum_{i=1}^{n}\varphi\Big[a_{n}\frac{X_{i}-u-n^{-1/2}z\hat{\sigma}}{\hat{\sigma}}\Big]> 0\Big)-[1-\Phi(z)]\Big|,\\
  T_{4}(z)&=\Big|\p\Big(\frac{1}{\sqrt{n}a_{n}}\sum_{i=1}^{n}\varphi\Big[a_{n}\frac{X_{i}-u-n^{-1/2}z\hat{\sigma}}{\hat{\sigma}}\Big]\geq  0\Big)-[1-\Phi(z)]\Big|.
\end{align*}
In what follows, we only bound $T_{3}(z)$ on $|z|\leq \sqrt{\ln n}$ since  argument for $T_{4}(z)$ is similar. By the inequality $|\varphi(x)-x|\leq x^{2}$ for any $x\in \R$ and basic inequality $(a+b)^{2}\leq  2a^{2}+2b^{2}$, we have
\begin{align*}
  &\Big|\frac{1}{\sqrt{n}a_{n}}\Big(\sum_{i=1}^{n}\varphi\Big[a_{n}\frac{X_{i}-u-n^{-1/2}z\hat{\sigma}}{\hat{\sigma}}\Big]
  -\frac{a_{n}(X_{i}-u-n^{-1/2}z\hat{\sigma})}{\hat{\sigma}}\Big)\Big|\nonumber\\
  &\leq \frac{2a_{n}}{\sqrt{n}}\sum_{i=1}^{n}\frac{(X_{i}-u)^{2}}{\hat{\sigma}^{2}}+2n^{-1/2}a_{n}z^{2}\nonumber\\
  &\leq \frac{2a_{n}}{\sqrt{n}}\Big(\frac{S_{n}}{\sqrt{n}\hat{\sigma}}\Big)^{2}+2\sqrt{n}a_{n}+2n^{-1/2}a_{n}z^{2}:=\Delta,
\end{align*}
where $S_{n}=\sum_{i=1}^{n}(X_{i}-u)$ and the last inequality follows from the equality 
\[
(n-1)\hat{\sigma}^{2}=\sum_{i=1}^{n}(X_{i}-u)^{2}-\frac{S_{n}^{2}}{n}.
\]
Noting that \[\Delta\leq 6\sqrt{n}a_{n}\quad  \text{ on } \Big\{\Big(\frac{S_{n}}{\sqrt{n}\hat{\sigma}}\Big)^{2}\leq n\Big\} \text{ and } |z|\leq \sqrt{\ln n},\]
then
\begin{align}\label{eq-self-mean--04}
T_{3}(z)\leq \max\{T_{5}(z), T_{6}(z)\}+\p\Big(\Big(\frac{S_{n}}{\sqrt{n}\hat{\sigma}}\Big)^{2}> n\Big),
\end{align}
where 
\begin{align*}
  T_{5}(z)&=\Big|\p\Big(\frac{S_{n}}{\sqrt{n}\hat{\sigma}}> z-6\sqrt{n}a_{n}\Big)-[1-\Phi(z)]\Big|,\\
   T_{6}(z)&=\Big|\p\Big(\frac{S_{n}}{\sqrt{n}\hat{\sigma}}> z+6\sqrt{n}a_{n}\Big)-[1-\Phi(z)]\Big|.
\end{align*}
By Theorem 1.1 in \cite{bentkus1996berry},  
\begin{align}\label{eq-self-mean--05}
\sup_{x\in \R}\Big|\p(S_{n}/(\sqrt{n}\hat{\sigma})> x)-[1-\Phi(x)]\Big|\leq C\beta_{2}+C\beta_{3}, 
\end{align}
which further implies that
\[
\max\{T_{5}(z), T_{6}(z)\}\leq C\beta_{2}+C\beta_{3}+C\sqrt{n}a_{n}.
\]
For the second term of (\ref{eq-self-mean--04}), by (\ref{eq-self-mean--05}) again, we have
\begin{align*}
 \p\Big(\frac{S_{n}}{\sqrt{n}\hat{\sigma}}> \sqrt{n}\Big)&\leq \sup_{x\in \R}\Big|\p\Big(\frac{S_{n}}{\sqrt{n}\hat{\sigma}}> x\Big)-[1-\Phi(x)]\Big|+1-\Phi(\sqrt{n})\nonumber\\
 &\leq C\beta_{2}+C\beta_{3}+Cn^{-1/2},
\end{align*}
 where the last inequality follows the inequality $1-\Phi(x)\leq e^{-x^{2}/2}$ for $x\geq 0.$
\end{proof}
\begin{proof}[Proof of Theorem \ref{thm-self-mean-md}]
 By Theorem \ref{thm-self-mean-clt} and Markov's inequality, we know that (\ref{eq-thm-mean-md-01}) holds for $0\leq z\leq2$.   
Now, consider $z\geq 2.$ 
By (\ref{eq-self-mean--01}), it suffices to prove that there  exist $c_{0},C\in \R_{+}$ such that for  $2\leq z\leq c_{0}\gamma_{n}d_{2+\delta}$
 \begin{align}
   &T_{3}\leq C[1-\Phi(z)]\big((1+z)^{2+\delta}n^{-\delta/2}+(1+z)\sqrt{n}a_{n}\big)d_{2+\delta}^{-(2+\delta)},\label{eq-mean-self-pr-1} \\
    &T_{4} \leq C[1-\Phi(z)]\big((1+z)^{2+\delta}n^{-\delta/2}+(1+z)\sqrt{n}a_{n}\big)d_{2+\delta}^{-(2+\delta)}.\label{eq-mean-self-pr-2}
 \end{align}
 We only prove (\ref{eq-mean-self-pr-1}) and (\ref{eq-mean-self-pr-2}) can be proven similarly. Further, by (\ref{eq-self-mean--04}), it suffices to provide upper bounds for $T_{5},T_{6}$ and $\p(S_{n}^{2}/(n\hat{\sigma}^{2})> n)$. Consider $2\leq z\leq \gamma_{n}d_{2+\delta}/4$. For $T_{5}$, let $ z_{s}=z-6\sqrt{n}a_{n}$, then 
  \begin{align*}
    T_{5}\leq \Big|\p\Big(\frac{S_{n}}{\sqrt{n}\hat{\sigma}}> z_{s}\Big)-[1-\Phi(z_{s})]\Big|+\big|\Phi(z_{s})-\Phi(z)\big|:=T_{51}+T_{52}.
  \end{align*}
  For $T_{52},$ by (\ref{eq-we-10}), we  have 
 \begin{align}\label{eq-mean-self-pr-4}
  T_{52}&\leq 2[1-\Phi(z)]\cdot
   z\cdot|z_{s}-z|e^{z\cdot |z_{s}-s|}\leq C[1-\Phi(z)]\cdot (1+z)\sqrt{n}a_{n} e^{6z\sqrt{n}a_{n}}\nonumber\\
   &\leq C[1-\Phi(z)]\cdot (1+z)\sqrt{n}a_{n}.
 \end{align}
 A straightforward consequence of (\ref{eq-mean-self-pr-4}) is 
  \begin{align}\label{eq-mean-self-pr-4.1}
  1-\Phi(z_{s})\leq C[1-\Phi(z)]\quad \text{ for } 2\leq z\leq \gamma_{n}d_{2+\delta}/4.
  \end{align}
 For $T_{52}$, we  shall apply  \citet[Theorem 2.3]{shao-and-jing(2003)} to bound $T_{52}$. To this end,  let $V_{n}^{2}=\sum_{i=1}^{n}(X_{i}-u)^{2}$ and $b_{n,x}=\big(n/(n+x^{2}-1)\big)^{1/2}$, one can check that 
 \begin{align*}
   \p\big(S_{n}/(\sqrt{n}\hat{\sigma})> x\big)=\p(S_{n}/V_{n}> xb_{n,x}),\quad\text{for all }x\geq 0,
 \end{align*}
 and then 
  \begin{align}\label{eq-mean-self-pr-6}
  T_{51}\leq  \big|\p\big(S_{n}/V_{n}> z'_{s}\big)-[1-\Phi(z'_{s})]\big|+\big|\Phi(z'_{s})-\Phi(z_{s})\big|:=T_{511}+T_{512},
 \end{align}
 where $z'_{s}=z_{s}b_{n,z_{s}}.$  Without loss of generality, we assume that $\sqrt{n}a_{n}\leq 1/6$, otherwise \eqref{eq-thm-mean-md-01} holds trivially.  This implies $6\sqrt{n}a_{n}\leq 1\leq z/2$  and hence 
 \begin{align}\label{eq-mean-self-pr-7}
   1\leq z/2\leq z_{s}\leq z\leq \gamma_{n}d_{2+\delta}/4.
 \end{align} In addition, note that $1/\sqrt{2}\leq b_{n,x}\leq \sqrt{2}$ for $0\leq x\leq \sqrt{n}$ and $d_{2+\delta}\leq 1, \gamma_{n}\leq \sqrt{n}$ for large $n$, then 
\begin{align}\label{eq-mean-self-pr-8}
  0\leq z'_{s}=z_{s}b_{n,z_{s}}\leq \gamma_{n}d_{2+\delta}.
 \end{align}
 Applying  \citet[Theorem 2.3]{shao-and-jing(2003)} yields that
 \begin{align}\label{eq-mean-self-pr-9}
 T_{511}\leq C[1-\Phi(z'_{s})](1+z'_{s})^{2+\delta}n^{-\delta/2}d_{2+\delta}^{-(2+\delta)}.
 \end{align}
 By (\ref{eq-mean-self-pr-9}), to bound $T_{511}$ and $T_{512}$, it suffices to bound $|\Phi(z'_{s})-\Phi(z_{s})|$. In fact, by  (\ref{eq-we-10}), we have 
  \begin{align}\label{eq-mean-self-pr-10}
  |\Phi(z'_{s})-\Phi(z_{s})|&\leq 2[1-\Phi(z_{s})]\cdot
   z_{s}\cdot|z'_{s}-z_{s}|e^{z_{s}\cdot |z'_{s}-z_{s}|}
  \end{align}
Moreover, by the definition of $z'_{s}$ and $b_{n,x}$ and (\ref{eq-mean-self-pr-7}), we have 
  \begin{align}\label{eq-mean-self-pr-11}
   |z'_{s}-z_{s}|&\leq z_{s}|b_{n,z_{s}}-1|=\frac{z_{s}}{b_{n,z_{s}}+1}\cdot \big|b_{n,z_{s}}^{2}-1\big|\nonumber\\
   &\leq 2z\cdot \frac{z_{s}^{2}-1}{n+z_{s}^{2}-1}\leq 4z^{3}n^{-1}.
 \end{align}
Combining the fact that $z\leq \gamma_{n}d_{2+\delta}/4\leq \gamma_{n}\leq n^{1/6}$ and (\ref{eq-mean-self-pr-4.1}), (\ref{eq-mean-self-pr-7}), (\ref{eq-mean-self-pr-10}), (\ref{eq-mean-self-pr-11}) yields that
 \begin{align}\label{eq-mean-self-pr-12}
 |\Phi(z'_{s})-\Phi(z_{s})|\leq C[1-\Phi(z)]\cdot
   z^{4}n^{-1}\leq C[1-\Phi(z)]\cdot
   (1+z^{2+\delta})n^{-\delta/2}
 \end{align}
 By (\ref{eq-mean-self-pr-4}), (\ref{eq-mean-self-pr-6}), (\ref{eq-mean-self-pr-9}) and (\ref{eq-mean-self-pr-12}), we have 
  \begin{align}\label{eq-mean-self-pr-13}
  T_{5}&\leq C[1-\Phi(z)]\cdot
   \big((1+z^{2+\delta})n^{-\delta/2}d_{2+\delta}^{-(2+\delta)}+(1+z)\sqrt{n}a_{n}\big)
  \end{align}
With a similar argument as that leading to (\ref{eq-mean-self-pr-13}), we can obtain the same upper bound for $T_{6}(z).$
 For $\p(S_{n}^{2}/(n\hat{\sigma}^{2})> n)$, by (\ref{eq-pq-05}) and Theorem 2.16 in \cite{shao2009self}, we have for any $x\geq 0$,
 \begin{align*}
\p(S_n / V_n>x)& \leq \p(V_n \leq \sqrt{n}\sigma/ 2)+\p(S_n>x V_n, V_n>\sqrt{n}\sigma / 2) \\
& \leq \p(V_n^2/n \leq \sigma^{2}/4)+\p(S_n>x(4 \sqrt{n}\sigma+V_n)/9)\nonumber\\
&\leq \exp\Big(-c_{\delta}nd_{2+\delta}^{(4+2\delta)/\delta}\Big)+2 \exp(-x^2/162), 
 \end{align*}
 where $c_{\delta}=\frac{\delta   }{(2+\delta)}(\frac{3}{4})^{(2+\delta)/\delta}.$ Then 
 \begin{align*}
   \p(S_{n}/(\sqrt{n}\hat{\sigma})> \sqrt{n})\leq \p(S_n / V_n>\sqrt{n/2})\leq  \exp\Big(-c_{\delta}nd_{2+\delta}^{(4+2\delta)/\delta}\Big)+2\exp(-n/324).
 \end{align*}
By (\ref{eq-we-11}), we have 
 \begin{align*}
 \exp\Big(-c_{\delta}nd_{2+\delta}^{(4+2\delta)/\delta}\Big)\leq C[1-\Phi(z)]z\exp\big(z^{2}/2-c_{\delta}nd_{2+\delta}^{(4+2\delta)/\delta}\big).
 \end{align*}
 Note that $z\leq \gamma_{n}d_{2+\delta}\leq O(n^{1/6})$, then for all  sufficiently large $n$, there exists $c_{1}$ depending on $\delta$ and $d_{2+\delta}$ such that 
 \[
 z^{2}/2-c_{\delta}nd_{2+\delta}^{(4+2\delta)/\delta}\leq -c_{1}nd_{2+\delta}^{(4+2\delta)/\delta},
 \]
 then 
 \begin{align*}
 \exp\Big(-c_{\delta}nd_{2+\delta}^{(4+2\delta)/\delta}\Big)&\leq C[1-\Phi(z)]zn^{-\delta/2}d_{2+\delta}^{-2-\delta}\nonumber\\
 &\leq C[1-\Phi(z)](1+z^{2+\delta})n^{-\delta/2}d_{2+\delta}^{-2-\delta}.
 \end{align*}
 By a similar argument, we can obtain similar upper bound for $\exp(-n/324)$, and hence 
  \begin{align}\label{eq-mean-self-pr-19}
\p(S_{n}/(\sqrt{n}\hat{\sigma})> \sqrt{n})&\leq C[1-\Phi(z)](1+z^{2+\delta})n^{-\delta/2}d_{2+\delta}^{-2-\delta}.
 \end{align}
Similarly, we have the same upper bound for $\p(S_{n}/(\sqrt{n}\hat{\sigma})< -\sqrt{n})$, then 
  \begin{align}\label{eq-mean-self-pr-20}
\p(S_{n}^{2}/(n\hat{\sigma}^{2})> n)&\leq C[1-\Phi(z)](1+z^{2+\delta})n^{-\delta/2}d_{2+\delta}^{-2-\delta}.
 \end{align}
 Combine (\ref{eq-mean-self-pr-13}) and (\ref{eq-mean-self-pr-20}), we obtain  (\ref{eq-mean-self-pr-1}).
\end{proof}
\section{Proofs of main results for Catoni-type robust regression}\label{section-proof-of-catoni-type}

\subsection{Proof of Theorem \ref{thm-nonasymtotic-theory}}
  We apply the generalized Poincar\'{e}-Miranda Theorem \cite[Theorem 2.4]{FRANKOWSKA2018832} to prove Theorem \ref{thm-nonasymtotic-theory}. To this end, let 
  \[
  \mathcal{K}=\big\{\bdbe\in \R^{d}; \|\bdbe-\bdbe^{*}\|_{2}\leq \beta_{0}\big\}. 
  \]
  Clearly, for any $\bdbe\in \partial\mathcal{K}$, the outward normal vector is $N_{\mathcal{K}}(\bdbe)=(\bdbe-\bdbe^{*})/\|\bdbe-\bdbe^{*}\|_{2}$, and hence we only need to show that with probability at least  $1-\varepsilon$, 
  \begin{align}\label{eq-low-log-pr-01}
   \langle h(\bdbe), \bdbe-\bdbe^{*} \rangle \leq 0,\quad \forall \bdbe\in \partial\mathcal{K}.
  \end{align}
 To prove (\ref{eq-low-log-pr-01}), define
 \begin{align*}
   r(\bdbe):=\frac{1}{n \alpha} \sum_{i=1}^n \frac{\bdx_{i}'(\bdbe-\bdbe^{*}) }{L_{n}\|\bdbe-\bdbe^{*}\|_2} \varphi[\alpha(y_{i}-\bdx_{i}'\bdbe)].
 \end{align*}
 By the inequality  $x\varphi(y)\leq |x|\log\big(1+\sgn(x)y+y^{2}/2\big)$ and Jensen's inequality, we have 
 \begin{align}\label{eq-low-log-pr-03}
  &\e[\exp(n\alpha r(\bdbe))]=\e\Big\{\prod_{i=1}^n \exp\Big(\frac{\bdx_{i}'(\bdbe-\bdbe^{*}) }{L_{n}\|\bdbe-\bdbe^{*}\|_2} \varphi[\alpha(y_{i}-\bdx_{i}'\bdbe)]\Big)\Big\}\nonumber\\
  &\leq \e\Big\{\prod_{i=1}^n \exp\Big(\frac{|\bdx_{i}'(\bdbe-\bdbe^{*})| }{L_{n}\|\bdbe-\bdbe^{*}\|_2}\log\Big(
  1+\alpha\cdot \mathrm{sign}[\bdx_{i}'(\bdbe-\bdbe^{*})](y_{i}-\bdx_{i}'\bdbe)\nonumber\\
  &\hspace*{5.5cm}+\frac{\alpha^{2}(y_{i}-\bdx_{i}'\bdbe)^{2}}{2}
  \Big)\Big\}\nonumber\\
  &=\prod_{i=1}^n\e\bigg\{\Big(
  1+\alpha\cdot \mathrm{sign}[\bdx_{i}'(\bdbe-\bdbe^{*})](y_{i}-\bdx_{i}'\bdbe)+\frac{\alpha^{2}(y_{i}-\bdx_{i}'\bdbe)^{2}}{2}
  \Big)^{\frac{|\bdx_{i}'(\bdbe-\bdbe^{*})| }{L_{n}\|\bdbe-\bdbe^{*}\|_2}}\bigg\}\nonumber\\
  &\leq 
  \prod_{i=1}^n\Big(
  1-\alpha|\bdx_{i}'(\bdbe^{*}-\bdbe)|+\frac{\alpha^{2}[\sigma_{i}^{2}+|\bdx_{i}'(\bdbe^{*}-\bdbe)|^{2}]}{2}
  \Big)^{\frac{|\bdx_{i}'(\bdbe-\bdbe^{*})| }{L_{n}\|\bdbe-\bdbe^{*}\|_2}}\nonumber\\
  &\leq \exp\bigg(\sum_{i=1}^{n}\frac{|\bdx_{i}'(\bdbe-\bdbe^{*})| }{L_{n}\|\bdbe-\bdbe^{*}\|_2}\Big(-\alpha|\bdx_{i}'(\bdbe^{*}-\bdbe)|+\frac{\alpha^{2}[\sigma_{i}^{2}+|\bdx_{i}'(\bdbe^{*}-\bdbe)|^{2}]}{2}  \Big)  \bigg),
 \end{align}
 where the last inequality uses the inequality $1+x\leq e^{x}$ for any $x\in \R.$ By (\ref{eq-low-log-pr-03}) and Markov's inequality, with probability at least $1-\varepsilon$, 
 \begin{align*}
  r(\bdbe)\leq \frac{1}{n}\sum_{i=1}^{n}\frac{|\bdx_{i}'(\bdbe-\bdbe^{*})| }{L_{n}\|\bdbe-\bdbe^{*}\|_2}\Big(-|\bdx_{i}'(\bdbe^{*}-\bdbe)|+\frac{\alpha[\sigma_{i}^{2}+|\bdx_{i}'(\bdbe^{*}-\bdbe)|^{2}]}{2}  \Big)+\frac{\log(\varepsilon^{-1})}{n\alpha}.
 \end{align*}
 Now consider $\bdbe\in \partial \mathcal{K}$, then
 \begin{align}\label{eq-low-log-pr-05}
    r(\bdbe)&\leq -\frac{1}{n}\sum_{i=1}^{n}\frac{|\bdx_{i}'(\bdbe-\bdbe^{*})|^{2} }{L_{n}\|\bdbe-\bdbe^{*}\|_2}+\frac{\alpha\bar{\sigma}^{2}}{2}+\frac{\alpha}{2n}\sum_{i=1}^{n}|\bdx_{i}'(\bdbe^{*}-\bdbe)|^{2}  +\frac{\log(\varepsilon^{-1})}{n\alpha}\nonumber\\
    &\leq -\frac{c_{l} }{L_{n}}\|\bdbe-\bdbe^{*}\|_{2} +\frac{\alpha\bar{\sigma}^{2}}{2}+\frac{\alpha c_{u}}{2}\|\bdbe^{*}-\bdbe\|_{2}^{2}  +\frac{\log(\varepsilon^{-1})}{n\alpha}\nonumber\\
    &=-\frac{c_{l}\beta_{0}}{L_{n}} +\frac{\alpha\bar{\sigma}^{2}}{2}+\frac{\alpha c_{u}\beta_{0}^{2} }{2} +\frac{\log(\varepsilon^{-1})}{n\alpha}\nonumber\\
    &=0.
 \end{align}
 This proves (\ref{eq-low-log-pr-01}).

\subsection{Proof of Theorem  \ref{thm-low-BE-bound}}
We begin with a lemma that is essential for the proof of Theorem  \ref{thm-low-BE-bound}.
For each $1\leq i\leq n$, define 
\begin{align*}
 \bar{\varepsilon}_{i}=\varepsilon_{i}\mathbf{1}(L_{n}\alpha_{n}|\varepsilon_{i}|\leq C'),\quad \text{where}\quad C'=\min\Big\{\frac{c_{l}}{\sqrt{c_{u}}20(K_{1}+K_{0})}, \frac{\sqrt{c_{l}}}{6}  \Big\}.
\end{align*}
Let $\bbdbe$ be the solution of the following equation:
\begin{align}\label{eq-mm-02}
    \frac{1}{n\alpha_{n}}\sum_{i=1}^{n}\bdx_{i}\varphi[\alpha_{n}(\bar{\varepsilon}_{i}+\bdx_{i}'(\bdbe^{*}-\bdbe))]=0.
\end{align}
We remark that if $\hbdbe$
  is viewed as a function of $(\varepsilon_{1},\varepsilon_{2},\ldots, \varepsilon_{n})$, then $\bbdbe$ is the same function evaluated at  $(\bar{\varepsilon}_{1},\bar{\varepsilon}_{2},\cdots, \bar{\varepsilon}_{n})$.
\begin{lem}\label{lem-moment-of-beta-beta-star} If Assumptions \ref{assu-low-dimension-gram-matrix} and \ref{assu-BE-low-demension} hold, then there exists a positive constant $C$ such that 
\begin{align}\label{eq-lem-moment-of-beta-pr-01}
  \|\boldsymbol{\bar{\beta}}-\bdbe^{*}\|_{2}\leq C\Big\|\frac{1}{n\alpha_{n}}\sum_{i=1}^{n}\bdx_{i}\big\{\varphi[\alpha_{n}\bar{\varepsilon}_{i}]-\e\varphi[\alpha_{n}\bar{\varepsilon}_{i}]\big\}\Big\|_{2}+C\sqrt{p}L_{n}a_{n}\sigma.
\end{align}
\end{lem}
\begin{proof} It follows from the definition of $\bbdbe$ that
\begin{align}\label{eq-mm-03}
\frac{1}{n}\sum_{i=1}^{n} \bdx_{i}\bdx_{i}'(\bbdbe-\bdbe^{*}) &=\frac{1}{n\alpha_{n}}\sum_{i=1}^{n}\bdx_{i}\varphi[\alpha_{n}(\bar{\varepsilon}_{i}+\bdx_{i}'(\bdbe^{*}-\bbdbe))]-\frac{1}{n}\sum_{i=1}^{n} \bdx_{i}\bdx_{i}'(\bdbe^{*}-\bbdbe) \nonumber\\
&:=\frac{1}{n\alpha_{n}}\sum_{i=1}^{n}\bdx_{i}\varphi[\alpha_{n}\bar{\varepsilon}_{i}]+R.
\end{align}
where 
\begin{align}\label{eq-mm-3.1}
   R:=\frac{1}{n\alpha_{n}}\sum_{i=1}^{n}\bdx_{i}\big(\varphi[\alpha_{n}(\bar{\varepsilon}_{i}+\bdx_{i}'(\bdbe^{*}-\bbdbe))]
-\varphi[\alpha_{n}\bar{\varepsilon}_{i}]-\alpha_{n}\bdx_{i}'(\bdbe^{*}-\bbdbe)\big).
\end{align}
Noting that 
\begin{align*}
\dot{\varphi}_{+}(0)&:=\lim_{x\to 0^{+}}\frac{\varphi(x)}{x}\leq \lim_{x\to 0^{+}}\frac{\log(1+x+x^{2}/2)}{x}=1,\nonumber\\
\text{and }&\quad\lim_{x\to 0^{+}}\frac{\varphi(x)}{x}\geq \lim_{x\to 0^{+}}\frac{-\log(1-x+x^{2}/2)}{x}=1,
\end{align*}
then $\dot{\varphi}_{+}(0)=1.$ Similarly, $\dot{\varphi}_{-}(0)=1$, then $\dot{\varphi}(0)=1.$  So
by Assumption \ref{assu-BE-low-demension} and mean value theorem, there exists $\xi\in (0,1)$ such that 
\begin{align}\label{eq-mm-04}
  &\big|\varphi[\alpha_{n}(\bar{\varepsilon}_{i}+\bdx_{i}'(\bdbe^{*}-\bbdbe))]
-\varphi[\alpha_{n}\bar{\varepsilon}_{i}]-\alpha_{n}\bdx_{i}'(\bdbe^{*}-\bbdbe)\big|\nonumber\\
&\quad=  \big|\big(\dot{\varphi}[\alpha_{n}\bar{\varepsilon}_{i}+\xi\alpha_{n}\bdx_{i}'(\bdbe^{*}-\bbdbe)]
-\dot{\varphi}(0)\big)\cdot\alpha_{n}\bdx_{i}'(\bdbe^{*}-\bbdbe)\big|\nonumber\\
&\quad\leq K_{1}\alpha_{n}^{2}|\bar{\varepsilon}_{i}||\bdx_{i}'(\bdbe^{*}-\bbdbe)|+K_{1}\alpha_{n}^{2}|\bdx_{i}'(\bdbe^{*}-\bbdbe)|^{2}.
\end{align}
In addition, by the condition $|\dot{\varphi}(x)|\leq  K_{0}$, we also have
\begin{align}\label{eq-mm-05}
 \big|\varphi[\alpha_{n}(\bar{\varepsilon}_{i}+\bdx_{i}'(\bdbe^{*}-\bbdbe))]
-\varphi[\alpha_{n}\bar{\varepsilon}_{i}]-\alpha_{n}\bdx_{i}'(\bdbe^{*}-\bbdbe)\big|&\leq 2K_{0}\alpha_{n}|\bdx_{i}'(\bdbe^{*}-\bbdbe)|.
\end{align}
Combine (\ref{eq-mm-04}) and (\ref{eq-mm-05}), we have 
\begin{align*}
&\big|\varphi[\alpha_{n}(\bar{\varepsilon}_{i}+\bdx_{i}'(\bdbe^{*}-\bbdbe))]
-\varphi[\alpha_{n}\bar{\varepsilon}_{i}]-\alpha_{n}\bdx_{i}'(\bdbe^{*}-\bbdbe)\big|\nonumber\\
&\leq \big[K_{1}\alpha_{n}^{2}|\bar{\varepsilon}_{i}||\bdx_{i}'(\bdbe^{*}-\bbdbe)|+K_{1}\alpha_{n}^{2}|\bdx_{i}'(\bdbe^{*}-\bbdbe)|^{2}\big]
\mathbf{1}(\alpha_{n}|\bdx_{i}'(\bdbe^{*}-\bbdbe)|\leq 1)\nonumber\\
&\quad+2K_{0}\alpha_{n}|\bdx_{i}'(\bdbe^{*}-\bbdbe)|\mathbf{1}(\alpha_{n}|\bdx_{i}'(\bdbe^{*}-\bbdbe)|\geq 1),
\end{align*}
which further implies that
\begin{align}\label{eq-mm-07}
\|R\|_{2}&\leq \frac{1}{n\alpha_{n}}\sum_{i=1}^{n}\|\bdx_{i}\|_{2}\cdot\big|\varphi[\alpha_{n}(\bar{\varepsilon}_{i}+\bdx_{i}'(\bdbe^{*}-\bbdbe))]
-\varphi[\alpha_{n}\bar{\varepsilon}_{i}]-\alpha_{n}\bdx_{i}'(\bdbe^{*}-\bbdbe)\big|\nonumber\\
&\leq \frac{L_{n}K_{1}}{n}\sum_{i=1}^{n}\alpha_{n}|\bar{\varepsilon}_{i}||\bdx_{i}'(\bdbe^{*}-\bbdbe)|+ \frac{L_{n}(2K_{0}+K_{1})}{n}\cdot
H_{1}
\end{align}
where 
\begin{align*}
  H_{1}= \sum_{i=1}^{n}\alpha_{n}|\bdx_{i}'(\bdbe^{*}-\bbdbe)|^{2}
\mathbf{1}(\alpha_{n}|\bdx_{i}'(\bdbe^{*}-\bbdbe)|\leq 1)+|\bdx_{i}'(\bdbe^{*}-\bbdbe)|\mathbf{1}(\alpha_{n}|\bdx_{i}'(\bdbe^{*}-\bbdbe)|\geq 1).
\end{align*}
We claim that 
\begin{align}\label{eq-mm-08}
H_{1}\leq 5\sum_{i=1}^{n}\alpha_{n}|\bar{\varepsilon}_{i}||\bdx_{i}'(\bdbe-\bdbe^{*})|.
\end{align}
By (\ref{eq-mm-07}) and (\ref{eq-mm-08}), together with the bound $L_{n}\alpha_{n}|\bar{\varepsilon}_{i}|\leq C'$, we have
\begin{align}\label{eq-mm-09}
 \|R\|_{2}&\leq \frac{10L_{n}(K_{1}+K_{0})}{n}\sum_{i=1}^{n}\alpha_{n}|\bar{\varepsilon}_{i}||\bdx_{i}'(\bdbe^{*}-\bbdbe)|\nonumber\\
 &\leq \frac{10(K_{1}+K_{0})C'}{n}\sum_{i=1}^{n}|\bdx_{i}'(\bdbe^{*}-\bbdbe)|\nonumber\\
 &\leq 10(K_{1}+K_{0})\sqrt{c_{u}}\cdot C'\|\bdbe^{*}-\bbdbe\|_{2}\nonumber\\
 &\leq 0.5c_{l}\|\bdbe^{*}-\bbdbe\|_{2},
\end{align}
where the penultimate inequality follows from 
\begin{align*}
 \sum_{i=1}^{n}|\bdx_{i}'(\bdbe^{*}-\bbdbe)|\leq \sqrt{n}\cdot\sqrt{\sum_{i=1}^{n}|\bdx_{i}'(\bdbe^{*}-\bbdbe)|^{2}}\leq n\sqrt{c_{u}}\|\bdbe^{*}-\bbdbe\|_{2}.
\end{align*}
Since $\bdS_{n}$ is a real symmetric matrix and $\lambda_{\min}(\bdS_{n})\geq c_{l}$, then $\|\bdS_{n} (\bbdbe-\bdbe^{*})\|_{2}\geq c_{l}\|\bbdbe-\bdbe^{*}\|_{2}$. Combining this with (\ref{eq-mm-03}) and (\ref{eq-mm-09}), we obtain 
\begin{align}\label{eq-mm-10}
  0.5c_{l}\|\bbdbe-\bdbe^{*}\|_{2}&\leq \Big\|\frac{1}{n\alpha_{n}}\sum_{i=1}^{n}\bdx_{i}\varphi[\alpha_{n}\bar{\varepsilon}_{i}]\Big\|_{2}.
\end{align}
By the definition of $\bar{\varepsilon}_{i}$ and inequality $|\varphi(x)-x|\leq x^{2}$, we have
\begin{align}\label{eq-mm-10.1}
\big|\e\varphi[\alpha_{n}\bar{\varepsilon}_{i}]\big|&= \big|\e \varphi[\alpha_{n}\bar{\varepsilon}_{i}]-\e \alpha_{n}\bar{\varepsilon}_{i}\big|+\alpha_{n}|\e\bar{\varepsilon}_{i}|\nonumber\\
&\leq \e\big| \varphi[\alpha_{n}\bar{\varepsilon}_{i}]- \alpha_{n}\bar{\varepsilon}_{i}\big|+ \alpha_{n}|\e\varepsilon_{i}\mathbf{1}(L_{n}\alpha_{n}|\varepsilon_{i}|\geq C')|\nonumber\\
 &\leq \alpha_{n}^{2}\e\varepsilon_{i}^{2}+(C')^{-1}L_{n}\alpha_{n}^{2}\e\varepsilon_{i}^{2}\leq CL_{n}\alpha_{n}^{2}\sigma_{i}^{2}.
\end{align}
Then (\ref{eq-lem-moment-of-beta-pr-01}) follows from (\ref{eq-mm-10}), (\ref{eq-mm-10.1}) and inequality $n^{-1}\sum_{i=1}^{n}\|\bdx_{i}\|_{2}^{2}=\text{trace}(\bdS_{n})\leq c_{u}p.$

Now, we prove (\ref{eq-mm-08}). To this end, 
 for each $i\in [n]$, define $$r_{i}(\bdbe):=\bdx_{i}'(\bdbe-\bdbe^{*})\varphi[\alpha_{n}(\bar{\varepsilon}_{i}+\bdx_{i}'(\bdbe^{*}-\bdbe))].$$ 
Denote $\varphi_{+}(x)=\log(1+x+x^{2}/2)$.  By (\ref{eq-t-04}) and  the non-decreasing property of $\varphi$, if $\bdx_{i}'(\bdbe-\bdbe^{*})\geq 0$, we have
\begin{align}\label{eq-mm-11}
     r_{i}(\bdbe)&\leq
     |\bdx_{i}'(\bdbe-\bdbe^{*})|\varphi\big[\max\{\alpha_{n}(\bar{\varepsilon}_{i}+\bdx_{i}'(\bdbe^{*}-\bdbe)),-1\}\big]\nonumber\\
     &\leq |\bdx_{i}'(\bdbe-\bdbe^{*})|\varphi_{+}\big[\max\{\alpha_{n}(\bar{\varepsilon}_{i}+\bdx_{i}'(\bdbe^{*}-\bdbe)),-1\}\big].
\end{align}
Similarly, if $\bdx_{i}'(\bdbe-\bdbe^{*})< 0$, then
\begin{align}\label{eq-mm-12}
     r_{i}(\bdbe)&\leq
     \bdx_{i}'(\bdbe-\bdbe^{*})\varphi\big[\min\{\alpha_{n}(\bar{\varepsilon}_{i}+\bdx_{i}'(\bdbe^{*}-\bdbe)),1\}\big]\nonumber\\
     &\leq |\bdx_{i}'(\bdbe-\bdbe^{*})|\varphi_{+}\big[-\min\{\alpha_{n}(\bar{\varepsilon}_{i}+\bdx_{i}'(\bdbe^{*}-\bdbe)),1\}\big]
\end{align}
Note that $\max\{x,y\}=\frac{|x-y|+x+y}{2}$. So combining (\ref{eq-mm-11}) and (\ref{eq-mm-12}) yields 
\begin{align}\label{eq-mm-13}
   r_{i}(\bdbe)&\leq  
   |\bdx_{i}'(\bdbe-\bdbe^{*})|\varphi_{+}\Big[ \frac{|z_{i}|+z_{i}}{2}-1\Big]\nonumber\\
   &\leq |\bdx_{i}'(\bdbe-\bdbe^{*})|\log\Big(\frac{1}{2}+\frac{1}{8}(|z_{i}|+z_{i})^{2}\Big),
\end{align}
where 
\[
z_{i}=\alpha_{n}\cdot\sgn(\bdx_{i}'(\bdbe-\bdbe^{*}))(\bar{\varepsilon}_{i}+\bdx_{i}(\bdbe^{*}-\bdbe))+1.
\]
If $z_{i}\leq 0$,  by (\ref{eq-mm-13}), we have
\begin{align}\label{eq-mm-14}
  r_{i}(\bdbe)\leq -\log(2) |\bdx_{i}'(\bdbe-\bdbe^{*})|.
\end{align}
If $z_{i}\geq 0$, by the definition of $z_{i}$, we have
\begin{align}\label{eq-mm-15}
  &\log\Big(\frac{1}{2}+\frac{1}{8}(|z_{i}|+z_{i})^{2}\Big)=\log\Big(\frac{1}{2}+\frac{1}{2}z_{i}^{2}\Big)\nonumber\\
  &=\log\Big(1+\alpha_{n}\big(\sgn(\bdx_{i}'(\bdbe-\bdbe^{*}))\bar{\varepsilon}_{i}-|\bdx_{i}'(\bdbe-\bdbe^{*})|\big)
  +\frac{\alpha_{n}^{2}}{2}\big(\bar{\varepsilon}_{i}+\bdx_{i}(\bdbe^{*}-\bdbe)\big)^{2}\Big)\nonumber\\
  &\leq \log\Big(1+\alpha_{n}|\bar{\varepsilon}_{i}|-\alpha_{n}|\bdx_{i}'(\bdbe-\bdbe^{*})|
  +\frac{\alpha_{n}^{2}\bar{\varepsilon}_{i}^{2}}{2}+\alpha_{n}^{2}|\bar{\varepsilon}_{i}|\cdot|\bdx_{i}(\bdbe^{*}-\bdbe)|+\frac{\alpha_{n}^{2}[\bdx_{i}(\bdbe^{*}-\bdbe)]^{2}}{2}\Big)\nonumber\\
  &\leq \log\Big(1+\alpha_{n}|\bar{\varepsilon}_{i}|-\alpha_{n}|\bdx_{i}'(\bdbe-\bdbe^{*})|/2
  +\alpha_{n}^{2}\bar{\varepsilon}_{i}^{2}/2+3\alpha_{n}^{2}|\bar{\varepsilon}_{i}|\cdot|\bdx_{i}(\bdbe^{*}-\bdbe)|/2
\Big):=I,
\end{align}
where the last inequality follows from the fact that if $z_{i}\geq 0$, then
\begin{align*}
   \alpha_{n}|\bdx_{i}(\bdbe^{*}-\bdbe)|\leq 1+\alpha_{n}|\bar{\varepsilon}_{i}|. 
\end{align*}
Since $\frac{1}{n}\sum_{i=1}^{n}\|\bdx_{i}\|_{2}^{2}=\text{trace}(\bdS_{n})\geq c_{l}p$ and $\sum_{i=1}^{n}\|\bdx_{i}\|_{2}^{2}\leq nL_{n}^{2}$, it follows that $L_{n}\geq \sqrt{c_{l}p},$ which further implies 
\[
\alpha_{n}|\bar{\varepsilon}_{i}|\leq \frac{C'}{L_{n}}\leq \frac{1}{6\sqrt{p}}\leq \frac{1}{6}.
\]
So by the inequalities $\log(1+x+x^{2}/2)\leq |x|$ and
$$\log(1+x+y)\leq \log(1+x)+\frac{y}{1+x},\quad x>-1,$$
we have
\begin{align}\label{eq-mm-17}
I&\leq \log\Big(1+\alpha_{n}|\bar{\varepsilon}_{i}|-\alpha_{n}|\bdx_{i}'(\bdbe-\bdbe^{*})|/4
  +\alpha_{n}^{2}\bar{\varepsilon}_{i}^{2}/2
\Big)\nonumber\\
  &\leq \log\big(1+\alpha_{n}|\bar{\varepsilon}_{i}|
  +\alpha_{n}^{2}\bar{\varepsilon}_{i}^{2}/2\big)-\frac{\alpha_{n}|\bdx_{i}'(\bdbe-\bdbe^{*})|}{4(1+\alpha_{n}|\bar{\varepsilon}_{i}|
+\alpha_{n}^{2}\bar{\varepsilon}_{i}^{2}/2)}\nonumber\\
  &\leq  \alpha_{n}|\bar{\varepsilon}_{i}|-\frac{1}{5}\alpha_{n}|\bdx_{i}'(\bdbe-\bdbe^{*})|.
\end{align}
Combining (\ref{eq-mm-15}) and (\ref{eq-mm-17}) yields that for $z_{i}\geq 0$,
\begin{align}\label{eq-mm-18}
  \log\Big(\frac{1}{2}+\frac{1}{8}(|z_{i}|+z_{i})^{2}\Big)&\leq \alpha_{n}|\bar{\varepsilon}_{i}|-\frac{1}{5}\alpha_{n}|\bdx_{i}'(\bdbe-\bdbe^{*})|.
\end{align}
It follows from (\ref{eq-mm-14}) and (\ref{eq-mm-18}) that
\begin{align}\label{eq-mm-19}
   r_{i}(\bdbe)&\leq \alpha_{n}|\bar{\varepsilon}_{i}||\bdx_{i}'(\bdbe-\bdbe^{*})|\mathbf{1}(z_{i}\geq 0)-0.2\alpha_{n}|\bdx_{i}'(\bdbe-\bdbe^{*})|^{2}\mathbf{1}(z_{i}\geq 0)\nonumber\\
   &\quad-\log(2)|\bdx_{i}'(\bdbe-\bdbe^{*})|\mathbf{1}(z_{i}\leq 0).
\end{align} 
If $\alpha_{n}|\bdx_{i}'(\bdbe-\bdbe^{*})|\geq 1$, then 
\begin{align}\label{eq-mm-20}
 & 0.2\alpha_{n}|\bdx_{i}'(\bdbe-\bdbe^{*})|^{2}\mathbf{1}(z_{i}\geq 0)+\log(2)|\bdx_{i}'(\bdbe-\bdbe^{*})|\mathbf{1}(z_{i}\leq 0)\nonumber\\
  &\geq 0.2|\bdx_{i}'(\bdbe-\bdbe^{*})|\mathbf{1}(z_{i}\geq 0)+\log(2)|\bdx_{i}'(\bdbe-\bdbe^{*})|\mathbf{1}(z_{i}\leq 0)\nonumber\\
  &\geq 0.2|\bdx_{i}'(\bdbe-\bdbe^{*})|.
\end{align}
If $\alpha_{n}|\bdx_{i}'(\bdbe-\bdbe^{*})|\leq 1$, then 
\begin{align}\label{eq-mm-21}
  & 0.2\alpha_{n}|\bdx_{i}'(\bdbe-\bdbe^{*})|^{2}\mathbf{1}(z_{i}\geq 0)+\log(2)|\bdx_{i}'(\bdbe-\bdbe^{*})|\mathbf{1}(z_{i}\leq 0)\nonumber\\
  &\geq 0.2|\bdx_{i}'(\bdbe-\bdbe^{*})|\mathbf{1}(z_{i}\geq 0)+\log(2)\alpha_{n}|\bdx_{i}'(\bdbe-\bdbe^{*})|^{2}\mathbf{1}(z_{i}\leq 0)\nonumber\\
  &\geq 0.2\alpha_{n}|\bdx_{i}'(\bdbe-\bdbe^{*})|^{2}.
\end{align}
Combining (\ref{eq-mm-19})--(\ref{eq-mm-21})  and the fact that  $\sum_{i=1}^{n}r_{i}(\bbdbe)=0$ yields that
\begin{align}\label{eq-mm-22}
&\sum_{i=1}^{n}|\bdx_{i}'(\bbdbe-\bdbe^{*})|\mathbf{1}(\alpha_{n}|\bdx_{i}'(\bbdbe-\bdbe^{*})|\geq 1)+\alpha_{n}|\bdx_{i}'(\bbdbe-\bdbe^{*})|^{2}\mathbf{1}(\alpha_{n}|\bdx_{i}'(\bbdbe-\bdbe^{*})|\leq 1)\nonumber\\
&\leq 5\sum_{i=1}^{n}\alpha_{n}|\bar{\varepsilon}_{i}||\bdx_{i}'(\bbdbe-\bdbe^{*})|,
\end{align}
and \eqref{eq-mm-08} follows.
\end{proof}
Now, we are ready to prove Theorem \ref{thm-low-BE-bound}.

\begin{proof}[Proof of Theorem \ref{thm-low-BE-bound}] 
 We shall apply  \citet[Corollary 2.2]{shao(2022)} to prove Theorem \ref{thm-low-BE-bound}. 
To this end, we first decompose $\sqrt{n}\bdS_{n}^{1/2}(\hbdbe-\bdbe^{*}-\boldsymbol{\delta}_{n})/\tilde{\sigma}$ into the sum of independent random vectors plus a remainder.
   By the definition of $\hbdbe$, we have 
\begin{align*}
   \bdS_{n} (\hbdbe-\bdbe^{*})
   &=\frac{1}{n\alpha_n}\sum_{i=1}^{n} \alpha_{n}\bdx_{i}\bdx_{i}'(\hbdbe-\bdbe^{*})\\
   &\ =
\frac{1}{n\alpha_{n}}\sum_{i=1}^{n}\bdx_{i}\big(\varphi[\alpha_{n}(\varepsilon_{i}+\bdx_{i}'(\bdbe^{*}-\hbdbe))]
-\varphi[\alpha_{n}\varepsilon_{i}]-\alpha_{n}\bdx_{i}'(\bdbe^{*}-\hbdbe)\big)\\
&\quad +\frac{1}{n\alpha_{n}}\sum_{i=1}^{n}\bdx_{i}\varphi[\alpha_{n}\varepsilon_{i}],
\end{align*}
which further implies that 
\begin{align*}
\frac{\sqrt{n}\bdS_{n}^{1/2}(\hbdbe-\bdbe^{*}-\boldsymbol{\delta}_{n})}{\tilde{\sigma}}=W+D,
\end{align*}
where 
\begin{align*}
W&=\frac{\bdS_{n}^{-1/2}}{\sqrt{n}\alpha_{n}\tilde{\sigma}}\sum_{i=1}^{n}\bdx_{i}\big\{\varphi[\alpha_{n}\varepsilon_{i}]-\e\varphi[\alpha_{n}\varepsilon_{i}]\big\}:=\sum_{i=1}^{n}\bdxi_{i},\\
  D&=\frac{\bdS_{n}^{-1/2}}{\sqrt{n}\alpha_{n}\tilde{\sigma}}\sum_{i=1}^{n}\bdx_{i}\big(\varphi[\alpha_{n}(\varepsilon_{i}+\bdx_{i}'(\bdbe^{*}-\hbdbe))]
-\varphi[\alpha_{n}\varepsilon_{i}]-\alpha_{n}\bdx_{i}'(\bdbe^{*}-\hbdbe)\big).
\end{align*} 
Note that $\e\{\bdxi_{i}\}=0$ for any $i\in [n]$ and 
\begin{align}\label{eq-kk-01z}
    \sum_{i=1}^{n}\e\big\{\bdxi_{i}\bdxi_{i}'\mid \bdx\big\}=\bdS_{n}^{-1/2}\cdot\frac{1}{n}
\sum_{i=1}^{n}\bdx_{i}\bdx'_{i}\cdot\bdS_{n}^{-1/2}=I_{p},
\end{align}
Therefore, the conditions of Corollary 2.2 in \cite{shao(2022)} are satisfied, and applying the Corollary yields 
\begin{align}\label{eq-k-01}
   &\sup_{A\in\mathcal{A}}\big|\p\big(\sqrt{n}\bdS_{n}^{1/2}(\hbdbe-\bdbe^{*}-\boldsymbol{\delta}_{n})/\tilde{\sigma}\in A \mid \bdx\big)-\p(Z\in A)\big|\nonumber\\
   &\qquad\quad\leq 259 p^{1 / 2} \sum_{i=1}^{n}\e\big\{\|\bdxi_{i}\|_{2}^{3}\mid \bdx\big\}+2 \mathbb{E}\{\|W\|_{2} \Delta\mid \bdx\}\nonumber\\
   &\qquad\qquad+2\sum_{i=1}^n \mathbb{E}\big\{\|\bdxi_i\|_{2}|\Delta-\Delta_{(i)}| \mid\bdx\big\}+\p(\mathcal{O}^{c}\mid \bdx)
\end{align}
for any measurable set $\mathcal{O}$ and  any random variables $\Delta$ and $(\Delta_{(i)})_{1 \leq i \leq n}$ such that $\Delta \geq\|D\|_{2}\mathbf{1}(\mathcal{O})$ and $\Delta_{(i)}$ is independent of $\bdxi_i$ conditional on $\bdx$.  
Choose 
\[
\mathcal{O}=\big\{ L_{n}\alpha_{n}|\varepsilon_{1}|\leq C', L_{n}\alpha_{n}|\varepsilon_{2}|\leq C', \cdots, L_{n}\alpha_{n}|\varepsilon_{n}|\leq C' \big\}.
\]
Then by  Markov's inequality, we have 
\begin{align}\label{eq-k-02}
    \p(\mathcal{O}^{c}\mid \bdx)\leq \sum_{i=1}^{n}\p(L_{n}\alpha_{n}|\varepsilon_{i}|\geq C'\mid \bdx)\leq CL_{n}^{3}n^{-1/2}\e|\varepsilon_{1}|^{3}/\sigma^{3}.
\end{align}
To choose a suitable $\Delta$, let us first give an upper bound for $\|D\|_{2}\mathbf{1}(\mathcal{O})$.
For any $1\leq i\leq n$, let $\bar{\varepsilon}_{i}=\varepsilon_{i}\mathbf{1}(L_{n}\alpha_{n}|\varepsilon_{i}|\leq C')$ and $\bbdbe$ be the solution of  equation (\ref{eq-mm-02}).  Under the event $\mathcal{O}$, it follows from (\ref{eq-mm-04})  that
  \begin{align}\label{eq-kk-01a}
      \|D\|_{2}&=\Big\| \frac{\bdS_{n}^{-1/2}}{\sqrt{n}\alpha_{n}\tilde{\sigma}}\sum_{i=1}^{n}\bdx_{i}\big(\varphi\big[\alpha_{n}(\bar{\varepsilon}_{i}+\bdx_{i}'(\bdbe^{*}-\bbdbe))\big]
  -\varphi[\alpha_{n}\bar{\varepsilon}_{i}]-\alpha_{n}[\bar{\varepsilon}_{i}+\bdx_{i}'(\bdbe^{*}-\bbdbe)]\big)\Big\|_{2}\nonumber\\
 &\leq \frac{c_{l}^{-1/2}}{\sqrt{n}\alpha_{n}\tilde{\sigma}}\Big\|\sum_{i=1}^{n}\bdx_{i}\big(\varphi\big[\alpha_{n}(\bar{\varepsilon}_{i}+\bdx_{i}'(\bdbe^{*}-\bbdbe))\big]
 -\varphi[\alpha_{n}\bar{\varepsilon}_{i}] -\alpha_{n}[\bar{\varepsilon}_{i}+\bdx_{i}'(\bdbe^{*}-\bbdbe)]\big)\Big\|_{2}\nonumber\\
  &=\frac{c_{l}^{-1/2}\sqrt{n}}{\tilde{\sigma}}\|R\|_{2},
  \end{align}
  with $R$ defined in \eqref{eq-mm-3.1}. Combining this with \eqref{eq-mm-09} and Lemma \ref{lem-moment-of-beta-beta-star} gives us, under the event $\mathcal{O}$, 
  \begin{align} \label{eq-kk-01a-2}
     \begin{split}
         \|D\|_{2}&\le \frac{CL_{n}}{\sqrt{n}\tilde{\sigma}}\sum_{i=1}^{n} \alpha_{n} |\bar{\varepsilon}_{i}|\bdx_{i}'(\bdbe^{*}-\bbdbe)|\\
      &\le \frac{CL_{n}\alpha_{n}}{\sqrt{n}\tilde{\sigma}}\sum_{i=1}^{n} |\varepsilon_{i}|\|\bdx_{i}\|_{2}\|\bdbe^{*}-\bbdbe\|_{2}\\
      &\le \frac{CL_{n}\alpha_{n}}{\sqrt{n}\tilde{\sigma}}\sum_{i=1}^{n} |\varepsilon_{i}|\|\bdx_{i}\|_{2}\bigg\{ \Big\|\frac{1}{n\alpha_{n}}\sum_{j=1}^{n}\bdx_{j}\big\{\varphi[\alpha_{n}\bar{\varepsilon}_{j}]-\e\varphi[\alpha_{n}\bar{\varepsilon}_{j}]\big\}\Big\|_{2}+\sqrt{p}L_{n}a_{n}\sigma\bigg\}.
     \end{split}
  \end{align}
   What's more, for $\tilde{\sigma}$, note that
   \begin{align}\label{eq-kk-01b}
      \begin{split}
          \big|\tilde{\sigma}^{2}/\sigma^{2}-1\big|
          &=\frac{1}{\sigma^{2}}\left|\operatorname{Var}(\varphi[\alpha_n \varepsilon_1])/\alpha_n^{2}-\sigma^{2}\right| \\
          &=\frac{1}{\sigma^{2}}\Big|\Big(\frac{1}{\alpha_n^{2}}\mathbb{E}[\varphi^{2}(\alpha_n \varepsilon_1)]-\sigma^{2}\Big)+\Big[\frac{1}{\alpha_n}\mathbb{E}\varphi(\alpha_n \varepsilon_1)\Big]^{2} \Big|.
      \end{split}
  \end{align}
  By the same argument as in (\ref{eq-t-3.1}), inequalities $|\varphi(x)|\leq |x|$, $|\varphi(x)-x|\leq |x|^{2}$ and noting that Assumption \ref{assu-mean-CLT} ensures $\alpha_{n}=a_n \sigma^{-1}\le C n^{-1/2}$, we have
  \begin{align*}
      \Big|\frac{1}{\alpha_n^{2}}\mathbb{E}[\varphi^{2}(\alpha_n \varepsilon_1)]-\sigma^{2}\Big|\le 4\alpha_{n} \mathbb{E}|\varepsilon_1|^{3}\le Cn^{-1/2}\sigma^{-1}\mathbb{E}|\varepsilon_1|^{3},
  \end{align*}
  and
  \begin{align*}
      \Big[\frac{1}{\alpha_n}\mathbb{E}\varphi(\alpha_n \varepsilon_1)\Big]^{2}
      &=\frac{1}{\alpha_n^{2}}\big[\mathbb{E}\{\varphi(\alpha_n \varepsilon_1)-\alpha_n\varepsilon_1\}\big]^{2}\\
      &\le \frac{1}{\alpha_{n}^{2}}\mathbb{E}\left|\varphi(\alpha_n \varepsilon_1)-\alpha_n\varepsilon_1\right|\cdot\mathbb{E}\{|\varphi(\alpha_n\varepsilon_1)|+\alpha_{n}|\varepsilon_{1}|\}\\
      &\le 2\alpha_{n} \mathbb{E}[\varepsilon_1^{2}]\cdot \mathbb{E}|\varepsilon_1|\le 2\alpha_{n}\mathbb{E}\left|\varepsilon_{1}\right|^{3}\le Cn^{-1/2}\sigma^{-1}\mathbb{E}|\varepsilon_1|^{3}.
  \end{align*}
  These together with \eqref{eq-kk-01b} gives us
  \begin{align} \label{eq-kk-01b-2}
      \big|\tilde{\sigma}^{2}/\sigma^{2}-1\big|\le c_{0}n^{-1/2}\frac{\mathbb{E}|\varepsilon_{1}|^{3}}{\sigma^{3}}
  \end{align}
 for some positive constant $c_{0}$.
  Without loss of generality, we assume that  \begin{align}\label{eq-kk-01c} pL_{n}^{3}n^{-1/2}\frac{\e|\varepsilon_{1}|^{3}}{\sigma^{3}}\leq 0.5\min\big\{c_{l}^{3/2}c_{0}^{-1}, 1\big\}.
  \end{align}
  Otherwise, (\ref{eq-thm-low-BE-bound-01}) holds with $C=2\max\big\{c_{0}c_{l}^{-3/2},1\big\}.$ By (\ref{eq-kk-01b-2}), (\ref{eq-kk-01c}) and the fact that $L_{n}\geq \sqrt{c_{l}p}$, we have $\tilde{\sigma}^{2}\geq 0.5\sigma^{2}.$ Combining this with (\ref{eq-kk-01a-2}) and the fact that $\alpha_{n}\le Cn^{-1/2}$, we obtain, under the event $\mathcal{O}$
  \begin{align}\label{eq-kk-01d}
      \begin{split}
          \|D\|_2&\leq \frac{CL_{n}\alpha_{n}}{\sqrt{n}\sigma}\sum_{i=1}^{n} |\varepsilon_{i}|\|\bdx_{i}\|_{2}\bigg\{ \Big\|\frac{1}{n\alpha_{n}}\sum_{j=1}^{n}\bdx_{j}\big\{\varphi[\alpha_{n}\bar{\varepsilon}_{j}]-\e\varphi[\alpha_{n}\bar{\varepsilon}_{j}]\big\}\Big\|_{2}+\sqrt{p}L_{n}a_{n}\sigma\bigg\}\\
      &\le c_1L_{n}T_{n}\big\{ \|M_{n}\|_{2}+\sqrt{p}L_{n}n^{-1/2}\big\},
      \end{split}
  \end{align}
  for some constant $c_1$, where 
  \[
Y_{i}=\frac{\varphi[\alpha_{n}\bar{\varepsilon}_{i}]-\e\varphi[\alpha_{n}\bar{\varepsilon}_{i}]}{\alpha_{n}},\quad M_{n}=\frac{1}{n\sigma}\sum_{i=1}^{n}\bdx_{i}Y_{i}\quad \text{and}\quad T_{n} = \frac{1}{n\sigma}\sum_{i=1}^{n}\|\bdx_{i}\|_{2}|\bar{\varepsilon}_{i}|.
\]
  Accordingly, we take 
  $$\Delta := c_1L_{n}T_{n}\big\{ \|M_{n}\|_{2}+\sqrt{p}L_{n}n^{-1/2}\big\} ,$$
   and $\Delta_{(i)}$ as the random variable obtained by substituting $\varepsilon_{i}$ in $\Delta$ with an independent copy $\varepsilon_i'$.

Up to this point, 
$\Delta$ and $\Delta_{(i)}$ have been well-defined. In what follows, we sequentially provide upper bounds for the first three terms  of (\ref{eq-k-01}). For convenience, in the following, we use $\e_{\bdx}[\,\cdot\,]$ to denote $\e [\,\cdot\mid \bdx]$. 

For the first term of (\ref{eq-k-01}), by the definition of $\bdxi_{i}$ and the inequality $|\varphi(x)|\leq |x|$, we have
\begin{align}\label{eq-tt-01}
  p^{1 / 2} \sum_{i=1}^{n}\e_{\bdx}\|\bdxi_{i}\|_{2}^{3}\leq Cp^{1/2}n^{-3/2}\frac{\e|\varepsilon_{1}|^{3}}{\sigma^{3}}\cdot \sum_{i=1}^{n}\|\bdx_{i}\|_{2}^{3}\leq Cp^{3/2}L_{n}n^{-1/2}\frac{\e|\varepsilon_{1}|^{3}}{\sigma^{3}},
\end{align} 
where the last inequality follows from 
\begin{align}\label{eq-kk-1.1}
    \frac{1}{n}\sum_{i=1}^{n}\|\bdx_{i}\|_{2}^{2}=\text{trace}(\bdS_{n})\leq c_{u}p.
\end{align}

For the second term of (\ref{eq-k-01}), by the Cauchy–Schwarz inequality, we obtain
\begin{align} \label{eq-aa-01}
   &\e_{\bdx}\{\|W\|_{2} \Delta\}\nonumber\\
    &\quad\le CL_{n} \big(\e_{\bdx}\{\|W\|_{2}^{2}\}\big)^{1/2} \big(\e_{\bdx}\{|T_{n}|^{4}\}\big)^{1/4}\big[\big(\e_{\bdx}\{\|M_{n}\|_{2}^{4}\}\big)^{1/4}+\sqrt{p}L_{n}n^{-1/2}\big].
\end{align}
It follows from the definition of $W$ and $M_{n}$, the inequality $|\varphi(x)|\le x$, $\tilde{\sigma}^{2}\ge 0.5\sigma^{2}$ and the independence that
\begin{align} \label{eq-aa-02}
    \e_{\bdx}\{\|W\|_{2}^{2}\}&=\frac{1}{n \alpha_{n}^{2} \tilde{\sigma}^{2}}\sum_{i=1}^{n} \|\bdS_{n}^{-1/2}\bdx_{i}\|_{2}^{2} \operatorname{Var}\left(\varphi[\alpha_{n}\varepsilon_{i}]\right)
    =p,
\end{align}
and
\begin{align}\label{eq-aa-03}
    \e_{\bdx}[M_{n}^{4}]&=\frac{1}{n^{4}\sigma^{4}}\e_{\bdx}\Big(\sum_{i=1}^{n}\|\bdx_{i}\|^{2}Y_{i}^{2}+\sum_{i\neq j=1}^{n}\bdx_{i}'\bdx_{j}Y_{i}Y_{j}\Big)\Big(\sum_{k=1}^{n}\|\bdx_{k}\|^{2}Y_{k}^{2}+\sum_{k\neq l=1}^{n}\bdx_{k}'\bdx_{l}Y_{k}Y_{l}\Big)\nonumber\\
    &=\frac{1}{n^{4}\sigma^{4}}\Big(\e_{\bdx}\Big(\sum_{i=1}^{n}\|\bdx_{i}\|^{2}Y_{i}^{2}\Big)^{2}+\sum_{i\neq j=1}^{n}\e_{\bdx}|\bdx_{i}'\bdx_{j}|^{2}Y_{i}^{2}Y_{j}^{2}\Big)\nonumber\\
    &\leq \frac{C}{n^{4}\sigma^{4}}\Big( \sum_{i=1}^{n}\e_{\bdx}\|\bdx_{i}\|^{4}\e|\bar{\varepsilon}_{i}|^{4}+\sum_{i\neq j=1}^{n}\|\bdx_{i}\|^{2}\|\bdx_{j}\|^{2}\sigma^{4} \Big)\nonumber\\
    &\leq \frac{C}{n^{4}} \Big( n\alpha_{n}^{-1}pL_{n}\e|\varepsilon_{1}|^{3}/\sigma^{2}+n^{2}p^{2} \Big)\nonumber\\
    &\leq Cn^{-2}p^{2} \e|\varepsilon_{1}|^{3}/\sigma^{3}.
\end{align}
For $\e_{\bdx}\{|T_{n}|^{4}\}$, by the Cauchy–Schwarz inequality and \eqref{eq-kk-1.1}, we have
\begin{align*}
    T_{n}\le \frac{1}{\sigma} \Big(\frac{1}{n}\sum_{i=1}^{n}\|\bdx_{i}\|_{2}^{2}\Big)^{1/2}\Big(\frac{1}{n}\sum_{i=1}^{n}\bar{\varepsilon}_{i}^{2}\Big)^{1/2}
    \le \frac{\sqrt{c_u p}}{\sigma}\Big(\frac{1}{n}\sum_{i=1}^{n}\bar{\varepsilon}_{i}^{2}\Big)^{1/2}.
\end{align*}
Hence, it then follows from \eqref{eq-kk-01c}, $|\bar{\varepsilon}_{i}|\le \frac{C'}{\alpha_{n}L_{n}}$ and $\alpha_{n}=a_n \sigma^{-1}\le Cn^{-1/2}$ that
\begin{align} \label{eq-aa-04}
   \begin{split}
        \e_{\bdx}\{|T_{n}|^{4}\}&\le \frac{c_{u}^{2}p^{2}} {n^{2}\sigma^{4}} \e\Big\{\Big(\sum_{i=1}^{n}\bar{\varepsilon}_i^{2}\Big)^{2}\Big\}
    \le \frac{c_{u}^{2}p^{2}} {n^{2}\sigma^{4}} \left(n(n-1)\sigma^{4}+n\e|\bar{\varepsilon}_1|^{4}\right)\\
    &\le \frac{Cp^{2}} {\sigma^{4}} \Big(\sigma^{4}+\frac{1}{\alpha_n L_{n}n}\e|\varepsilon_1|^{3}\Big)
    \le Cp^{2}.
   \end{split}
\end{align}
Now, combining \eqref{eq-aa-01}, \eqref{eq-aa-02}, \eqref{eq-aa-03} and \eqref{eq-aa-04}, we obtain
\begin{align} \label{eq-tt-02}
    \e_{\bdx}\{\|W\|_{2}\Delta\}\le \frac{CL_{n}^{2}p^{3/2}\e|\varepsilon_1|^{3}}{\sqrt{n}\sigma^{3}}.
\end{align}

As for the third term of (\ref{eq-k-01}), let $Y_{i}'$ $M_n^{(i)}$ and $T_n^{(i)}$ be the versions of $Y_i$, $M_n$ and $T_n$ in which $\varepsilon_i$ is replaced by an independent copy $\varepsilon_i'$. Then, it follows from the definitions of $\Delta$ and $\Delta_i$ that,  for $1\le i\le n$,
\begin{align*}
    \begin{split}
        |\Delta-\Delta_i|&=c_1L_{n}\left|\sqrt{p}L_{n}n^{-1/2}(T_{n}-T_{n}^{(i)})+T_{n}\|M_{n}\|_{2}-T_{n}^{(i)}\|M_{n}^{(i)}\|_{2} \right|\\
        &\le c_1L_{n}\Big\{\sqrt{p}L_{n}n^{-1/2}|T_{n}-T_{n}^{(i)}|+\big|T_n-T_{n}^{(i)}\big|\|M_{n}\|_{2}\\
        &\hspace*{5.5cm}+|T_{n}^{(i)}|\Big|\|M_{n}\|_{2}-\|M_{n}^{(i)}\|_{2}\Big| \Big\}.
    \end{split}
\end{align*}
Applying this with the H\"older inequality, we obtain, for $1\le i\le n$
\begin{align} \label{eq-aa-05}
     \begin{split}
        & \e_{\bdx}\{\|\bdxi_i\|_{2}|\Delta-\Delta_i|\}\\
    &\le CL_{n}\left(\e_{\bdx}\|\bdxi_i\|_{2}^{3} \right)^{1/3}\bigg\{\sqrt{p}L_{n}n^{-1/2}\left(\e_{\bdx}|T_{n}-T_{n}^{(i)}|^{3}\right)^{1/3}\\
    &\hspace*{3.5cm}+ \left(\e_{\bdx}|T_{n}-T_{n}^{(i)}|^{3}\right)^{1/3} \left(\e_{\bdx}\|M_{n}\|_{2}^{3} \right)^{1/3}\\
    &\hspace*{3.5cm}  +\left(\e_{\bdx}|T_{n}^{(i)}|^{3}\right)^{1/3} \Big(\e_{\bdx}\Big|\|M_{n}\|_{2}-\|M_{n}^{(i)}\|_{2}\Big|^{3} \Big)^{1/3}  \bigg\} .
     \end{split}
\end{align}
Since $\{\varepsilon_i\}_{1\le i\le n}$ are identically distributed, by the definition of $\bdxi_i$, the property $|\varphi(x)|\le |x|$ and $\tilde{\sigma}^{2}\ge 0.5 \sigma^{2}$, we obtain
\begin{align} \label{eq-aa-06}
    \e_{\bdx}\|\bdxi_{i}\|_{2}^{3}
    &\le \frac{1}{n^{3/2}\alpha_{n}^{3}\tilde{\sigma}^{3}}\|\bdS_n^{-1/2}\bdx_i\|_{2}^{3}\e\left|\varphi[\alpha_{n}\varepsilon_1]\right|^{3}  
    \le \frac{C\|\bdx_{i}\|_{2}^{3}\e|\varepsilon_1|^{3}}{n^{3/2}\sigma^{3}}.
\end{align}
At the same time, by definitions and $|\varphi(x)|\le |x|$, we have
\begin{align} \label{eq-aa-07}
    \e_{\bdx}\big|T_{n}^{(i)}-T_{n}\big|^{3}=\frac{1}{n^{3}\sigma^{3}}\|\bdx_i\|_{2}^{3}\e\left||\bar{\varepsilon}_{i}'|-|\bar{\varepsilon}_i|\right|^{3}
    &\le \frac{C\|\bdx_{i}\|_{2}^{3}\e|\varepsilon_1|^{3}}{n^{3}\sigma^{3}} ,
\end{align}
and
\begin{align} \label{eq-aa-08}
    \begin{split}
        \e_{\bdx}\big|\|M_{n}\|_{2}-\|M_{n}^{(i)}\|_{2}\big|^{3}
    &\le \e_{\bdx}\|M_{n}-M_{n}^{(i)}\|_{2}^{3}  \\ 
    &=\e_{\bdx}\Big\|\frac{1}{n\alpha_{n}\sigma}\bdx_i\left\{\varphi[\alpha_n\bar{\varepsilon}_i]-\varphi[\alpha_n\bar{\varepsilon}_i']\right\}\Big\|_{2}^{3}\\
    &\le \frac{C\|\bdx_{i}\|_{2}^{3}}{n^{3}}\e|\bar{\varepsilon}_i^{3}|\le \frac{C\|\bdx_{i}\|_{2}^{3}\e|\varepsilon_1|^{3}}{n^{3}\sigma^{3}}.
    \end{split}
\end{align}
Substituting \eqref{eq-aa-03}--\eqref{eq-aa-04} and \eqref{eq-aa-06}--\eqref{eq-aa-08} into \eqref{eq-aa-05}, and using \eqref{eq-kk-1.1}, we obtain
\begin{align}\label{eq-aa-09}
\sum_{i=1}^{n} \e_{\bdx}\!\left[\|\bdxi_i\|_{2}\,|\Delta-\Delta_i|\right]
\le Cp^{3/2}L_{n}^{2}n^{-1/2}\frac{\e|\varepsilon_1|^{3}}{\sigma^{3}}.
\end{align}

Finally, noting that $L_{n}\ge \sqrt{c_l p}$, the proof is completed by combining the estimates for the four terms in the upper bound of (\ref{eq-k-01}) obtained in \eqref{eq-k-02}, \eqref{eq-tt-01}, \eqref{eq-tt-02} and \eqref{eq-aa-09}.
\end{proof}

\section*{Acknowledgements}
 We are grateful to Professor Qiman Shao for carefully reading the manuscript and for providing valuable suggestions for improvement.

\bibliographystyle{apalike}
\bibliography{refs}

\end{document}